\documentclass[reqno]{amsart}
\usepackage{amsmath, amsthm, amscd, amssymb, amsfonts, amsbsy}
\usepackage{latexsym, color, enumerate}
\usepackage{pxfonts}
\usepackage[dvipsnames]{xcolor}
\usepackage{bbm}

\usepackage{marginnote}

\theoremstyle{plain}
\newtheorem{theorem}[equation]{Theorem}
\newtheorem{lemma}[equation]{Lemma}

\theoremstyle{definition}

\theoremstyle{remark}

\newcommand{\dv}{\operatorname{div}}

\newcommand{\dist}{\operatorname{dist}}

\newcommand{\tr}{\operatorname{tr}}

\newcommand{\Sym}{\operatorname{Sym}}
\numberwithin{equation}{section}

\newcommand{\bR}{\mathbb{R}}

\providecommand{\abs}[1]{\lvert#1\rvert}
\providecommand{\Abs}[1]{\left\lvert#1\right\rvert}
\providecommand{\bigabs}[1]{\bigl\lvert#1\bigr\rvert}

\providecommand{\norm}[1]{\lVert#1\rVert}

\providecommand{\bignorm}[1]{\bigl\lVert#1\bigr\rVert}

\renewcommand{\vec}[1]{\boldsymbol{#1}}

\begin{document}
\title[Regularity of solutions on critical sets]
{Refined regularity at critical points for linear elliptic equations}

\author[J. Choi]{Jongkeun Choi}
\address[J. Choi]{Department of Mathematics Education, Pusan National University,  Busan, 46241, Republic of Korea}
\email{jongkeun\_choi@pusan.ac.kr}
\thanks{J. Choi was supported by the National Research Foundation of Korea (NRF) under agreement NRF-2022R1F1A1074461.}

\author[H. Dong]{Hongjie Dong}
\address[H. Dong]{Division of Applied Mathematics, Brown University,
182 George Street, Providence, RI 02912, United States of America}
\email{Hongjie\_Dong@brown.edu}
\thanks{H. Dong was partially supported by the NSF under agreement DMS-2055244.}

\author[S. Kim]{Seick Kim}
\address[S. Kim]{Department of Mathematics, Yonsei University, 50 Yonsei-ro, Seodaemun-gu, Seoul 03722, Republic of Korea}
\email{kimseick@yonsei.ac.kr}
\thanks{S. Kim is supported by the National Research Foundation of Korea (NRF) under agreement NRF-2022R1A2C1003322.}

\subjclass[2010]{Primary 35B45, 35B65,  35J08}

\keywords{}

\begin{abstract}
We investigate the regularity of solutions to linear elliptic equations in both divergence and non-divergence forms, particularly when the principal coefficients have Dini mean oscillation. We show that if a solution $u$ to a divergence-form equation satisfies $Du(x^o)=0$ at a point, then the second derivative $D^2u(x^o)$ exists and satisfies sharp continuity estimates.
As a consequence, we obtain ``$C^{2,\alpha}$ regularity'' at critical points when the coefficients of $L$ are $C^\alpha$.
This result refines a theorem of Teixeira (Math. Ann. 358 (2014), no. 1--2, 241--256) in the linear setting, where both linear and nonlinear equations were considered.
We also establish an analogous result for equations in non-divergence form.
\end{abstract}

\maketitle

\section{Introduction and main results}
We consider the elliptic operator $L$ in divergence form:
\[
Lu=  D_i(a^{ij}D_ju) + b^i D_iu + cu=\dv (\mathbf A D u)+\vec b\cdot Du + cu,
\]
and the corresponding operator $\mathrm{L}$ in non-divergence form:
\[
\mathrm{L}u= a^{ij} D_{ij}u + b^i D_i  u + cu=\tr(\mathbf A D^2 u)+ \vec b \cdot Du +cu
\]
defined on a domain $\Omega \subset \bR^d$, where $d \ge 2$.
Throughout this article, we adopt the standard summation convention over repeated indices.

For a function $f$ defined on $\Omega\subset \bR^d$, we define its mean oscillation function $\omega_f(\cdot)$ by
\[
\omega_f(r):=\sup_{x\in \Omega} \fint_{\Omega\cap B_r(x)} \Abs{f-(f)_{\Omega \cap B_r(x)}},\quad (f)_{\Omega \cap B_r(x)}:=\fint_{\Omega\cap B_r(x)}f.
\]
We say that $f$ has Dini mean oscillation (abbreviated as DMO), and write $f \in \mathrm{DMO}$ if $\omega_f$ satisfies the Dini condition:
\[
\int_0^1 \frac{\omega_f(r)}{r}\,dr <\infty.
\]
It is clear that if $f$ is Dini continuous (i.e., its modulus of continuity satisfies the Dini condition), then $f$ has Dini mean oscillation.
However, the converse is not true: the DMO condition is strictly weaker than Dini continuity; see \cite{DK17} for examples.
Note that the DMO condition implies uniform continuity, with a modulus of continuity controlled by $\omega_f$; see the appendix of \cite{HK20}. 

This article establishes that if $u$ is a solution to the divergence-form equation $Lu=0$, and its first derivative vanishes at a point $x^o$ (i.e., $Du(x^o)=0$), then the second derivative $D^2u(x^o)$ exists and satisfies quantitative estimates at that point.
This conclusion is derived under the assumption that the principal coefficient matrix  $\mathbf A\in \mathrm{DMO}$, together with suitable conditions on the lower-order coefficients of $L$.
In particular, if  $\mathbf A \in C^\alpha $, $\vec b \in L^p$ with $p \ge d/(1-\alpha)$, and $c \in C^\alpha$  for some $\alpha \in (0,1)$, then $u$ enjoys ``$C^{2,\alpha}$ regularity'' at any critical point $x^o$.
That is, the second derivative $D^2u(x^o)$ exists, and $u$ satisfies the following estimate for $x$ sufficiently close to $x^o$:
\[
\abs{Du(x)-D^2u(x^o)(x-x^o)} \lesssim \abs{x-x^o}^{1+\alpha}.
\]
This result improves upon a notable theorem of Teixeira \cite{Teixeira} in the linear setting, which, under the assumption $\mathbf A \in C^\alpha$, $\vec b=0$, and $c=0$, establishes ``$C^{1,1^-}$ regularity'' at $x^o$.
That is, for every $0<\gamma<1$, the solution satisfies the estimate
\[
\abs{u(x)-u(x^o)} \lesssim \abs{x-x^o}^{1+\gamma}
\]
for $x$ sufficiently close to $x^o$.

An analogous result holds for the non-divergence form equation $\mathrm{L}u=0$.
Under the assumption that $\mathbf A \in \mathrm{DMO}$ and suitable regularity conditions on the lower-order coefficients, if $u$ is a solution of $\mathrm{L}u=0$ and satisfies $D^2u(x^o)=0$ at some point $x^o$, then the third derivative $D^3u(x^o)$ exists and satisfies appropriate continuity estimates. In particular, if $\mathbf A \in C^{\alpha}$ and $\vec b$, $c \in C^{1,\alpha}$ for some $\alpha \in (0,1)$, then $u$ enjoys ``$C^{3,\alpha}$ regularity''  at any point $x^o$ where its Hessian vanishes.

This article builds upon our recent work in \cite{CDKK24}, which investigated improved regularity for solutions of the double divergence form equation $\mathrm{L}^*u=0$ (the formal adjoint of the non-divergence form equation $\mathrm{L}u=0$) on the set where the solution $u$ vanishes.
In \cite{CDKK24}, we specifically established ``$C^{1,\alpha}$ regularity'' for solutions of $\mathrm{L}^*u=0$ at points $x^o$ where $u(x^o)=0$, under the assumptions that $\mathbf A \in C^\alpha$, $\vec b \in L^p$ for some $p\ge d/(1-\alpha)$, and and $c \in L^p$ for some $p\ge d/(2-\alpha)$.
This result refines an earlier notable theorem by Leit\~ao, Pimentel, and Santos \cite{LPS20}. 

\smallskip
We now present our main results.
In the following theorem, for $\xi \in \bR^d$, the notation $D^2u(x^o)\xi$ denotes the product of the symmetric $d \times d$ matrix $D^2u(x^o)$ and the vector $\xi$, regarded as a $d \times 1$ column vector. 
We define
\[
\abs{D^2u(x^o)}=\sup_{\abs{\xi}=1} \,\abs{D^2 u(x^o)\xi \cdot \xi}=\sup_{\abs{\xi}=1} \,\abs{D_{ij}u(x^o)\xi^i \xi^j}.
\]
Since $D^2u(x^o)$ is symmetric, this definition coincides with the usual operator norm and is equivalent to other matrix norms.

\begin{theorem} \label{thm01}
Let $\Omega \subset \bR^d$ be a domain.
Assume the matrix $\mathbf A=(a^{ij})$ satisfies the following ellipticity and boundedness conditions for some constants $\lambda, \Lambda >0$:
\begin{equation} \label{ellipticity-d}
\lambda |\xi|^2 \le a^{ij}(x) \xi^i \xi^j \quad \text{and} \quad |a^{ij}(x) \xi^i \eta^j| \le \Lambda |\xi| |\eta| \quad \text{for all } x \in \Omega,\; \xi, \eta \in \mathbb{R}^d.
\end{equation}
Furthermore, assume $\mathbf A \in \mathrm{DMO}$, $\vec b \in L^p$ for some $p>d$, and $c \in \mathrm{DMO}$.
Let $u \in C^1_{\rm loc}(\Omega)$ be a weak solution to the equation:
\begin{equation}		\label{eq_main-d}
Lu=\dv (\mathbf{A} D u) + \boldsymbol{b} \cdot Du + cu = f \quad \text{in }\; \Omega,
\end{equation}
where $f \in \mathrm{DMO}$.
If $Du(x^o)=0$ for some $x^o\in\Omega$, then $Du$ is differentiable at $x^o$.
More precisely, there exist positive constants $r_0$ and $C$, depending only on $d$, $\lambda$, $\Lambda$, and the coefficients of $L$, such that for $r:=\min\{r_0, \frac12 \dist(x^o,\partial\Omega)\}$, the second derivative $D^2u$ at $x^o$ satisfies the estimate:
\[
\abs{D^2 u (x^o)} \le C \left(\frac{1}{r} \fint_{B_{2r}(x^o)} \abs{Du}+ \varrho_{\rm lot}(r) +\varrho_f(r) \right).
\]
Here, the terms $\varrho_{\cdots}(\cdot)$ represent moduli of continuity, where 
$\varrho_{\rm coef}(\cdot)$ is determined by the coefficients of $L$, $\varrho_{\rm lot}(\cdot)$ depends on $\norm{u}_{L^\infty(B_r(x^o))}$ and $c$, and $\varrho_f(\cdot)$ is determined by $f$.
Furthermore, for $0<\abs{x-x^o}<r$, the following estimate holds:
\begin{multline*}		
\frac{\Abs{Du(x)-D^2u(x^o)(x-x^o)}}{\abs{x-x^o}}
\le C \varrho_{\rm coef}(\abs{x-x^o}) \left\{\frac{1}{r} \fint_{B_{2r}(x^o)} \abs{Du} + \varrho_{\rm lot}(r) +\varrho_{f}(r) \right\}\\
+C \varrho_{\rm lot}(\abs{x-x^o})+C \varrho_f(\abs{x-x^o}).
\end{multline*}
\end{theorem}

\smallskip

Here are a few quick remarks concerning Theorem \ref{thm01}:
First, under the assumptions of the theorem, any weak solution $u$ belongs to  $C^1_{\rm loc}(\Omega)$.
In fact, it suffices to assume $c \in L^p$ for some $p>d$, rather than requiring $c \in \mathrm{DMO}$.
For further details, refer to Theorem 1.3 in \cite{DEK18} and Theorem 1.5 in \cite{DK17}.
Second, the precise dependence of the moduli of continuity $\rho_{\cdots}(t)$ on the relevant parameters is described in Section \ref{sec3}. Specifically:
\begin{itemize}
\item
If $\mathbf A \in C^\alpha$ for some $\alpha \in (0,1)$ and $\vec b \in L^p$ with $p \ge d/(1-\alpha)$, then $\varrho_{\rm coef}(t) \lesssim t^\alpha$.
\item
If $c \in C^\alpha$ for some $\alpha \in (0,1)$, then  $\varrho_{\rm lot}(t) \lesssim t^\alpha$.
\item
If $f \in C^\alpha$ for some $\alpha \in (0,1)$, then  $\varrho_f(t) \lesssim t^\alpha$.
\end{itemize}
Also, we note that $\varrho_{\rm lot} \equiv 0$ if $c$ is constant, and $\varrho_f \equiv 0$ if $f$ is constant.
Finally, while it may seem more natural to consider the equation
\[
\dv (\mathbf A D u +\tilde{\vec b}u)+\vec b\cdot Du + cu=\dv \tilde{\vec f} + f
\]
instead of \eqref{eq_main-d}, doing so would require additional regularity assumptions on $\tilde{\vec b}$ and $\tilde{\vec f}$ to ensure that the equation reduces properly to \eqref{eq_main-d}.

\medskip
In the following theorem, for $\xi \in \bR^d$, the notation $\langle D^3u(x^o), \xi \rangle$ denotes a $d\times d$ matrix whose $(i,j)$-th entry is given by:
\[
\langle D^3u(x^o), \xi \rangle_{i,j}=D_{ijk}u(x^o)\xi^k.
\]
We also define the norm of the third-order derivative tensor as
\[
\abs{D^3u(x^o)}=\sup_{\abs{\xi}=1} \,\abs{D_{ijk}u(x^o)\xi^i \xi^j \xi^k}.
\]
This is equivalent (up to a constant depending only on the dimension) to the supremum of $\abs{\langle D^3u(x^o), e \rangle}$, where the norm is taken in the sense of operator norm for matrices, over all unit vectors $e \in \bR^d$.

\begin{theorem}			\label{thm02}
Let $\Omega \subset \bR^d$ be a domain.
Consider a symmetric matrix $\mathbf A=(a^{ij})$ satisfying the uniform ellipticity condition:
\begin{equation}					\label{ellipticity-nd}
\lambda \abs{\xi}^2 \le a^{ij}(x) \xi^i \xi^j \le \Lambda \abs{\xi}^2 \quad\text{for all }\; x\in\Omega,\;\; \xi\in\mathbb{R}^d,
\end{equation}
where $0<\lambda \le \Lambda$ are constants.
Additionally, we assume that $\mathbf A$, $\vec b$, $D \vec b$, $c$ and $Dc$ all belong to $\mathrm{DMO}$.
Let $u \in C^2_{\rm loc}(\Omega)$ be a solution of the equation
\begin{equation}		\label{eq_main-nd}
\mathrm{L}u=\tr (\mathbf A D^2u)+\vec b\cdot Du + cu= f \; \text{ in }\;\Omega,
\end{equation}
where $f$ and $Df$ also belong to $\mathrm{DMO}$.
If $D^2u(x^o)=0$ for some $x^o\in\Omega$, then $D^2u$ is differentiable at $x^o$.
More precisely, there exist positive constants $r_0$ and $C$, depending only on $d$, $\lambda$, $\Lambda$, and the coefficients of $\mathrm{L}$, such that for $r:=\min\{r_0, \frac12 \dist(x^o,\partial\Omega)\}$, the third derivative $D^3u$ at $x^o$ satisfies the estimate:
\[
\abs{D^3 u (x^o)} \le C \left(\frac{1}{r} \fint_{B_{2r}(x^o)} \abs{D^2u}+ \varrho_{\rm lot}(r) + \varrho_{Df}(r) \right).
\]
Here, the terms $\varrho_{\cdots}(\cdot)$ represent moduli of continuity: 
$\varrho_{\rm coef}(\cdot)$ is determined by the coefficients $\mathbf A$, $\vec b$, and $c$; $\varrho_{\rm lot}(\cdot)$ depends on $\norm{u}_{L^\infty(B_r(x^o))}$, $\norm{Du}_{L^\infty(B_r(x^o))}$, $D\vec b$, and $Dc$; and $\varrho_{Df}(\cdot)$ is determined by $Df$.
Furthermore, for $0<\abs{x-x^o}<r$, the following estimate holds:
\begin{multline*}
\frac{\Abs{D^2u(x)-\langle D^3u(x^o), x-x^o \rangle }}{\abs{x-x^o}}
\le C\varrho_{\rm coef}(\abs{x-x^o}) \left\{\frac{1}{r} \fint_{B_{2r}(x^o)} \abs{D^2u} + \varrho_{\rm lot}(r) +\varrho_{Df}(r) \right\}\\
+C \varrho_{\rm lot}(\abs{x-x^o})+C \varrho_{Df}(\abs{x-x^o}).
\end{multline*}
\end{theorem}

\smallskip
Here are a few brief remarks concerning Theorem \ref{thm02}:
First, under the assumptions of the theorem, any strong solution $u$ belongs to  $C^2_{\rm loc}(\Omega)$.
In fact, it is not necessary to assume that $D\vec b$, $Dc$, and $Df$ belong to $\mathrm{DMO}$; see \cite[Theorem 1.5]{DEK18} and \cite[Theorem 1.6]{DK17}.
Second, the precise dependence of the moduli of continuity $\rho_{\cdots}(t)$ on the parameters specified in the theorem is discussed in detail in Section \ref{sec5}. Specifically:
\begin{itemize}
\item
If $\mathbf A \in C^\alpha$ for some $\alpha \in (0,1)$, then $\varrho_{\rm coef}(t) \lesssim t^\alpha$.
\item
If $\vec b \in C^{1,\alpha}$ and $c \in C^{1,\alpha}$ for some $\alpha \in (0,1)$, then $\varrho_{\rm lot}(t) \lesssim t^\alpha$.
\item
If $f \in C^{1,\alpha}$ for some $\alpha \in (0,1)$, then  $\varrho_{Df}(t) \lesssim t^\alpha$.
\end{itemize}
Additionally, we note that $\varrho_{\rm lot} \equiv 0$ if $D\vec b$ and $c$ are constant, and $\varrho_{Df} \equiv 0$ if $Df$ is constant.

\smallskip

The proof of Theorem \ref{thm01} is presented in Sections \ref{sec2} and \ref{sec3}. Section \ref{sec2} considers a simplified setting in which the lower-order coefficients $\vec b$ and $c$ are absent, and the inhomogeneous term $f$ vanishes. This preliminary case is intended to highlight the core ideas without the distraction of technical complications. The full proof of Theorem \ref{thm01} in the general setting is then developed in Section \ref{sec3}. Similarly, the proof of Theorem \ref{thm02} is given in Sections \ref{sec4} and \ref{sec5}. Section \ref{sec4} addresses the special case where $\vec b$, $c$, and $f$ are all zero, while Section \ref{sec5} treats the general case in full detail.

\section{Proof of Theorem \ref{thm01}: Simple case}			\label{sec2}
In this section, we consider an elliptic operator $L$ in the form
\[
Lu=  D_i(a^{ij} D_ju)=\dv (\mathbf A Du).
\]
We also assume that the inhomogeneous term $f$ is identically zero.
These assumptions allow us to present the main idea of the proof more transparently.

We aim to show that if  $Du(x^o)=0$ for some point $x^o \in \Omega$, then $Du$ is differentiable at $x^o$.
For simplicity, we assume, without loss of generality, that $x^o=0$ and $\dist(x^o, \partial\Omega) \ge 1$. This ensures $B_{1}(0) \subset \Omega$.

Let  $\bar{\mathbf A}:=(\mathbf A)_{B_r}$ denote the average of $\mathbf A$ over the ball $B_r=B_r(0)$, where $r \in (0,\frac12]$.
We decompose $u$ as $u=v+w$, where $w \in W^{1,p}_0(B_r)$ (for some $p>1$) is the weak solution to the problem
\[
\dv(\bar{\mathbf A}Dw) = -\dv((\mathbf{A}-\bar{\mathbf A})Du)\;\mbox{ in }\; B_r,\quad
w=0  \;\mbox{ on }\;\partial B_r.
\]
\begin{lemma}			\label{lem01}
Let $B=B_1(0)$. Let $\mathbf A_0$ be a constant matrix satisfying condition \eqref{ellipticity-d}. For $\vec f \in L^p(B)$ and $g \in L^p(B)$ with some $p>1$, let $u \in W^{1,p}_0(B)$ be the weak solution to the problem:
\[
\left\{
\begin{aligned}
\dv(\mathbf A_0 Du)&= \dv \vec f+ g\;\mbox{ in }\; B,\\
u &= 0 \;\mbox{ on } \; \partial B.
\end{aligned}
\right.
\]
Then, for any $t>0$, we have
\[
\Abs{\{x \in B : \abs{Du(x)} > t\}} \le \frac{C}{t} \left(\int_{B} \abs{\vec f}+ \abs{g} \right),
\]
where $C=C(d, \lambda, \Lambda)$.
\end{lemma}

\begin{proof}
Refer to the proof of \cite[Lemma 3.3]{CDKK24} and \cite[Lemma 2.2]{DK17}.
\end{proof}

By Lemma \ref{lem01}, we have $Dw \in L^{\frac12}(B_r)$ and
\begin{equation}			\label{eq1622thu}
\left(\fint_{B_r} \abs{Dw}^{\frac12}\right)^{2} \le C \omega_{\mathbf A}(r) \norm{Du}_{L^\infty(B_r)},
\end{equation}
where $C=C(d,\lambda, \Lambda)$.
Although $Dw \in L^p(B_r)$ for any $p \in (0,1)$, we use $p=\frac12$ for simplicity. For more details, refer to the proof of Theorem 1.5 in \cite{DK17}.

Note that $v=u-w$ satisfies the following equation:
\[
\dv(\bar{\mathbf A} Dv)=0\quad\text{in }\;B_r.
\]
By interior regularity estimates for solutions of elliptic equations with constant coefficients, we know that $v \in C^{\infty} (B_r)$.
In particular, we have
\begin{equation}			\label{eq1232thu}
\norm{D^2v}_{L^\infty(B_{r/2})} \le \frac{C}{r^2} \left(\fint_{B_r} \abs{v}^{\frac12}\right)^{2},
\end{equation}
where $C$ is a constant depending only on $d$, $\lambda$, and $\Lambda$.
Moreover, this estimate remains valid if $v$ is replaced by $Dv-\vec l$, where $\vec l$ is any affine function of the form:
\[
\vec l(x)=x_1 \vec c_1 + \cdots +x_d \vec c_d + \vec p,\quad \text{with }\;\vec c_1,\ldots, \vec c_d, \vec p \in \bR^d. 
\]
We define the set $\mathfrak{A}$ as
\[
\mathfrak{A}=\left\{\vec l(x)=\mathbf{S}x + \vec p: \mathbf S \in \mathbb{S}^d,\; \vec p \in \mathbb{R}^d \right\},
\]
where $\mathbb{S}^d$ denotes the set of all real symmetric $d \times d$ matrices.

Therefore, for any $\vec l \in \mathfrak{A}$, we derive from \eqref{eq1232thu} that
\begin{equation}			\label{eq1228thu}
\norm{D^3v}_{L^\infty(B_{r/2})} \le \frac{C}{r^2} \left(\fint_{B_r} \abs{Dv-\vec l}^{\frac12}\right)^{2}.
\end{equation}

On the other hand, by Taylor's theorem, for any $\rho \in (0, r]$, we have
\[
\sup_{x\in B_\rho}\;\abs{Dv(x)-D^2v(0)x - Dv(0)}  \le C(d) \norm{D^3v}_{L^\infty(B_\rho)} \,\rho^2.
\]

Consequently, for any $\kappa \in (0, \frac12)$ and  $\vec l \in \mathfrak{A}$, we obtain
\begin{equation}		\label{eq1616fri}
\left(\fint_{B_{\kappa r}} \abs{Dv-D^2v(0)x -Dv(0)}^{\frac12}\right)^{2} \le C \norm{D^3v}_{L^\infty(B_{r/2})} (\kappa r)^2 \le C \kappa^2 \left(\fint_{B_r} \abs{Dv-\vec l}^{\frac12}\right)^{2},
\end{equation}
where $C=C(d,\lambda, \Lambda)>0$.

Now, define the function
\begin{equation}			\label{eq1100sat}
\varphi(r):=\frac{1}{r} \,\inf_{\vec l \in \mathfrak{A}} \left(\fint_{B_r} \abs{Du-\vec l}^{\frac12}\right)^{2}.
\end{equation}
Since $u=v+w$, we apply the quasi-triangle inequality, together with estimates \eqref{eq1622thu} and \eqref{eq1616fri}, to obtain
\begin{align}			\nonumber
\kappa r \varphi(\kappa r) & \le \left(\fint_{B_{\kappa r}} \abs{Du-D^2v(0)x-Dv(0)}^{\frac12}\right)^{2}\\					\nonumber
&\le C\left(\fint_{B_{\kappa r}} \abs{Dv-D^2v(0)x-Dv(0)}^{\frac12}\right)^{2} + C \left(\fint_{B_{\kappa r}} \abs{Dw}^{\frac12}\right)^{2}\\
					\nonumber
& \le  C \kappa^2 \left(\fint_{B_r} \abs{Dv-\vec l}^{\frac12}\right)^{2} + C \kappa^{-2d}\left(\fint_{B_{r}} \abs{Dw}^{\frac12}\right)^{2}\\	
					\nonumber
& \le C \kappa^2 \left(\fint_{B_r} \abs{Du-\vec l}^{\frac12}\right)^{2} + C \left(\kappa^2+\kappa^{-2d}\right)\left(\fint_{B_{r}} \abs{Dw}^{\frac12}\right)^{2}\\
					\label{eq0224sat}
& \le C \kappa^2 \left(\fint_{B_r} \abs{Du-\vec l}^{\frac12}\right)^{2} + C \left(\kappa^2+\kappa^{-2d}\right) \omega_{\mathbf A}(r) \norm{Du}_{L^\infty(B_r)}.
\end{align}

Let $\beta \in (0,1)$ be an arbitrary but fixed number.
With this $\beta$, choose $\kappa=\kappa(d, \lambda,\Lambda, \beta) \in (0, \frac12)$ such that $C \kappa \le  \kappa^{\beta}$.
Then, inequality \eqref{eq0224sat} yields
\begin{equation}			\label{eq1110sat}
\varphi(\kappa r) \le \kappa^\beta \varphi(r) + C \omega_{\mathbf A}(r)\,\frac{1}{r} \norm{Du}_{L^\infty(B_r)},
\end{equation}
where $C=C(d, \lambda, \Lambda, \kappa)=C(d, \lambda, \Lambda, \beta)$.

Let $r_0 \in (0, \frac12]$ be a number to be chosen later.
By iterating \eqref{eq1110sat}, we obtain, for each $j=1,2,\ldots$,
\[
\varphi(\kappa^j r_0) \le \kappa^{\beta j} \varphi(r_0) + C \sum_{i=1}^{j} \kappa^{(i-1)\beta} \omega_{\mathbf A}(\kappa^{j-i} r_0)\,\frac{1}{\kappa^{j-i} r_0}\, \norm{Du}_{L^\infty(B_{\kappa^{j-i} r_0})}.
\]

Define
\begin{equation}			\label{eq1244sat}
M_j(r_0):=\max_{0\le i < j}\, \frac{1}{\kappa^i r_0} \,\norm{Du}_{L^\infty(B_{\kappa^i r_0})} \quad \text{for }j=1,2, \ldots,
\end{equation}
so that we can estimate
\begin{equation}			\label{eq1114sat}
\varphi(\kappa^j r_0) \le \kappa^{\beta j} \varphi(r_0) + C M_j(r_0) \tilde \omega_{\mathbf A}(\kappa^j r_0),
\end{equation}
where, following \cite[(2.15)]{DK17}, we define
\begin{equation}			\label{eq1245sat}
\tilde\omega_{\mathbf A}(t):= \sum_{i=1}^\infty \kappa^{i\beta} \left\{ \omega_{\mathbf A}(\kappa^{-i} t) [ \kappa^{-i} t \le 1] +\omega_{\mathbf A}(1)[\kappa^{-i} t >1]\right\}.
\end{equation}
Here, we adopt Iverson bracket notation: for any statement $P$, we set $[P] = 1$ if $P$ is true, and $[P] = 0$ otherwise.
We emphasize that $\tilde\omega_{\mathbf A}(t)$ satisfies the Dini condition whenever $\omega_{\mathbf A}(t)$ does.
In particular, if $\mathbf{A} \in C^\alpha$ for some $\alpha \in (0,1)$, then choosing $\beta \in (\alpha,1)$ yields $\tilde\omega_{\mathbf{A}}(t) \lesssim t^\alpha$ for all $t \in (0,1]$.
Moreover, we note that $\omega_{\mathbf A}(t) \lesssim \tilde \omega_{\mathbf A}(t)$ for $t \in (0,1]$.
Additionally, we will use the following estimate (see \cite[Lemma 2.7]{DK17}):
\begin{equation}				\label{rmk1147}
\sum_{j=0}^\infty \tilde \omega_{\mathbf A}(\kappa^j r) \lesssim \int_0^{r} \frac{\tilde\omega_{\mathbf A}(t)}{t}\,dt,
\end{equation}
which, in particular, implies
\begin{equation}				\label{rmk1147_2}
\tilde \omega_{\mathbf A}(r) \lesssim \int_0^r \frac{\tilde \omega_{\mathbf A}(t)}{t}\,dt.
\end{equation}

For any fixed $r$, the infimum in \eqref{eq1100sat} is achieved by some $\vec l \in \mathfrak{A}$.
For $j=0,1,2,\ldots$, let $\mathbf{S}_j \in \mathbb{S}^d$ and $\vec p_j \in \bR^d$ be such that:
\begin{equation}			\label{eq0203tue}
\varphi(\kappa^j r_0)= \frac{1}{\kappa^j r_0}
\left(\fint_{B_{\kappa^j r_0}} \bigabs{Du(x)-\mathbf{S}_j  x -\vec p_j}^{\frac12}dx\right)^{2}.
\end{equation}

From \eqref{eq1100sat} and H\"older's inequality, we have
\begin{equation}			\label{eq0538tue}
\varphi(r_0) \le \frac{1}{r_0} \fint_{B_{r_0}} \abs{Du}.
\end{equation}
Combining \eqref{eq1114sat} with \eqref{eq0538tue}, we obtain the estimate
\begin{equation}\label{eq4.38}
\varphi(\kappa^j r_0) \le  \frac{\kappa^{\beta j}}{r_0} \fint_{B_{r_0}} \abs{Du}+ CM_j(r_0) \tilde \omega_{\mathbf A}(\kappa^j r_0).
\end{equation}

Next, observe that for each $j=0,1,2,\ldots$, we have
\begin{equation}\label{eq0218tue}
\fint_{B_{\kappa^j r_0}} \bigabs{\mathbf{S}_j  x +\vec p_j}^{\frac12}dx
\le \fint_{B_{\kappa^j r_0}} \bigabs{Du-\mathbf{S}_j x -\vec p_j}^{\frac12}dx+
\fint_{B_{\kappa^j r_0}} \abs{Du}^{\frac12}\le 2\fint_{B_{\kappa^j r_0}} \abs{Du}^{\frac12}.
\end{equation}
Furthermore, we observe that
\[
\bigabs{\vec p_j}^{\frac12} = \bigabs{\mathbf{S}_j x + \vec p_j - 2(\mathbf{S}_j (x/2) + \vec p_j)}^{\frac12} \le  \bigabs{\mathbf{S}_j x + \vec p_j}^{\frac12} + 2^{\frac12}\, \bigabs{\mathbf{S}_j (x/2) + \vec p_j}^{\frac12}.
\]
In addition, we have
\[
\fint_{B_{\kappa^j r_0}} \bigabs{\mathbf{S}_j (x/2) + \vec p_j}^{\frac12}dx  = \frac{2^d}{\abs{B_{\kappa^j r_0}}}\int_{B_{\kappa^j r_0/2}} \bigabs{\mathbf{S}_j x + \vec p_j}^{\frac12} dx \le 2^d\fint_{B_{\kappa^j r_0}} \bigabs{\mathbf{S}_j x + \vec p_j}^{\frac12}dx.
\]
Combining these estimates, we deduce
\begin{equation}		\label{eq0213tue}
\bigabs{\vec p_j} \le C \left(\fint_{B_{\kappa^j r_0}}\bigabs{\mathbf{S}_j x + \vec p_j }^{\frac12}\right)^{2}\le C \left(\fint_{B_{\kappa^j r_0}} \abs{Du}^{\frac12}\right)^{2},\quad j=0,1,2,\ldots.
\end{equation}
Since $u \in C^1(\overline B_{1/2})$ by \cite[Theorem 1.5]{DK17} and $Du(0)=0$, the estimate \eqref{eq0213tue} immediately yields
\begin{equation}			\label{eq0215tue}
\lim_{j \to \infty} \vec p_j =0.
\end{equation}

\subsection*{Estimate of $\mathbf{S}_j$}
By the quasi-triangle inequality, we have
\[
\bigabs{(\mathbf{S}_j - \mathbf{S}_{j-1})x + (\vec p_j -\vec p_{j-1})}^{\frac12} \le \bigabs{Du-\mathbf{S}_j x- \vec p_j}^{\frac12} + \bigabs{Du-\mathbf{S}_{j-1} x- \vec p_{j-1}}^{\frac12}.
\]
Taking the average over $B_{\kappa^j r_0}$ and using the fact that $\abs{B_{\kappa^{j-1} r_0}}/ \abs{B_{\kappa^j r_0}} = \kappa^{-d}$, we obtain
\begin{equation}			\label{eq1949sat}
\frac{1}{\kappa^j r_0}
\left(\fint_{B_{\kappa^j r_0}}\bigabs{(\mathbf{S}_j - \mathbf{S}_{j-1})x + (\vec p_j -\vec p_{j-1})}^{\frac12}dx\right)^{2} \le C \varphi(\kappa^j r_0) + C \varphi(\kappa^{j-1} r_0)
\end{equation}
for $j=1,2,\ldots$, where $C=C(d, \lambda, \Lambda, \beta)$.

Next, observe that
\[
\bigabs{\vec p_j-\vec p_{j-1}}^{\frac12} = \bigabs{(\mathbf{S}_j -\mathbf{S}_{j-1})x + (\vec p_j-\vec p_{j-1}) - 2\left((\mathbf{S}_j -\mathbf{S}_{j-1})(x/2) + (\vec p_j-\vec p_{j-1})\right)}^{\frac12}.
\]
Using the quasi-triangle inequality, we obtain
\[
\bigabs{\vec p_j-\vec p_{j-1}}^{\frac12} \le \bigabs{(\mathbf{S}_j -\mathbf{S}_{j-1})x + (\vec p_j-\vec p_{j-1})}^{\frac12} + 2^{\frac12} \,\bigabs{(\mathbf{S}_j -\mathbf{S}_{j-1})(x/2) + (\vec p_j-\vec p_{j-1})}^{\frac12}.
\]
Moreover,
\[
\fint_{B_{\kappa^j r_0}} \abs{(\mathbf S_j -\mathbf S_{j-1})(x/2) + (\vec p_j-\vec p_{j-1})}^{\frac12}dx \leq
2^d\fint_{B_{\kappa^j r_0}} \abs{(\mathbf S_j -\mathbf S_{j-1})x + (\vec p_j-\vec p_{j-1})}^{\frac12}dx.
\]
Combining these estimates, we conclude that
\[
\bigabs{\vec p_j-\vec p_{j-1}} \le C(d) \left(\fint_{B_{\kappa^j r_0}}\bigabs{(\mathbf{S}_j - \mathbf{S}_{j-1})x + (\vec p_j -\vec p_{j-1})}^{\frac12}dx\right)^{2},\quad j=1,2,\ldots.
\]
Substituting this into \eqref{eq1949sat}, we derive
\begin{equation}				\label{eq2247sat}
\frac{1}{\kappa^j r_0} \bigabs{\vec p_j - \vec p_{j-1}} \le C \varphi(\kappa^j r_0)+C \varphi(\kappa^{j-1} r_0),\quad j=1,2,\ldots.
\end{equation}

On the other hand, for any matrix $\mathbf{S} \in \mathbb{S}^d$, we may write $\mathbf{S} = \abs{\mathbf{S}} \mathbf{U}$, where $\mathbf{U} \in \mathbb{S}^d$ satisfies $\abs{\mathbf{U}} = 1$.
Then we have
\begin{equation}		\label{eq0229tue}
\fint_{B_r} \abs{\mathbf{S} x}^{\frac12}dx  \ge \abs{\mathbf{S}}^{\frac12} \inf_{\abs{\mathbf U}=1}  \fint_{B_r}  \abs{\mathbf U x}^{\frac12} = \abs{\mathbf{S}}^{\frac12} \inf_{\abs{\mathbf U}=1} \fint_{B_1} r^{\frac12} \abs{\mathbf U x}^{\frac12} =C(d) r^{\frac12} \abs{\mathbf{S}}^{\frac12}.
\end{equation}

Then, by using \eqref{eq0229tue}, the quasi-triangle inequality, \eqref{eq1949sat}, \eqref{eq2247sat},  \eqref{eq4.38}, and the observation that $M_{j-1}(r_0) \le M_{j}(r_0)$, we obtain
\begin{align}
						\nonumber
\abs{\mathbf{S}_j-\mathbf{S}_{j-1}} &\le \frac{C}{\kappa^j r_0} \left(\fint_{B_{\kappa^j r_0}}\bigabs{(\mathbf{S}_j - \mathbf{S}_{j-1})x}^{\frac12}dx\right)^{2}\\
						\nonumber
& \le \frac{C}{\kappa^j r_0} \left(\fint_{B_{\kappa^j r_0}}\bigabs{(\mathbf{S}_j - \mathbf{S}_{j-1})x + (\vec p_j -\vec p_{j-1})}^{\frac12}dx\right)^{2} + \frac{C}{\kappa^j r_0} \bigabs{\vec p_j - \vec p_{j-1}}\\
						\nonumber
&\le C \varphi(\kappa^j r_0)+C \varphi(\kappa^{j-1} r_0)\\
								\label{eq2316sat}
&\le \frac{C\kappa^{\beta j}}{r_0}  \fint_{B_{r_0}} \abs{Du}+ CM_j(r_0)\left\{\tilde \omega_{\mathbf A}(\kappa^j r_0) + \tilde \omega_{\mathbf A}(\kappa^{j-1} r_0)\right\}.
\end{align}

To estimate $\abs{\mathbf{S}_0}$, we proceed similarly to \eqref{eq2316sat} by applying \eqref{eq0218tue} and \eqref{eq0213tue} with $j=0$, and using H\"older's inequality to obtain
\begin{equation}			\label{eq0230tue}
\abs{\mathbf{S}_0} \le \frac{C}{r_0} \fint_{B_{r_0}} \abs{Du},
\end{equation}
where $C=C(d, \lambda, \Lambda, \beta)$.

For $k>l\ge 0$, we derive from \eqref{eq2316sat} and the definition of $M_{j}(r_0)$ that
\begin{align}
						\nonumber
\abs{\mathbf{S}_k -\mathbf{S}_l} &\le \sum_{j=l}^{k-1}\, \abs{\mathbf{S}_{j+1}-\mathbf{S}_j} \le  \sum_{j=l}^{k-1}  \frac{C\kappa^{\beta(j+1)}}{r_0} \fint_{B_{r_0}} \abs{Du}+ CM_k(r_0) \sum_{j=l}^{k} \tilde \omega_{\mathbf A}(\kappa^j r_0) \\
						\label{eq0946tue}
&\le  \frac{C \kappa^{\beta(l+1)}}{(1-\kappa^\beta)r_0} \fint_{B_{r_0}} \abs{Du}+ CM_k(r_0) \int_0^{\kappa^l r_0} \frac{\tilde \omega_{\mathbf A}(t)}{t}\,dt,
\end{align}
where we used \eqref{rmk1147}.
In particular, by taking $k=j$ and $l=0$ in \eqref{eq0946tue}, and using \eqref{eq0230tue}, we obtain for $j=1,2,\ldots$ that
\begin{equation}			\label{eq0900tue}
\abs{\mathbf{S}_j} \le \abs{\mathbf{S}_j-\mathbf{S}_0} + \abs{\mathbf{S}_0} \le \frac{C}{r_0} \fint_{B_{r_0}} \abs{Du}+ CM_j(r_0) \int_0^{r_0} \frac{\tilde \omega_{\mathbf A}(t)}{t}\,dt.
\end{equation}

Similarly, we obtain from \eqref{eq2247sat} that for $k>l\ge 0$, we have
\begin{equation}			\label{eq0947tue}
\bigabs{\vec p_k -\vec p_l} \le C \frac{\kappa^{(\beta+1)(l+1)}}{1-\kappa^{\beta+1}} \fint_{B_{r_0}} \abs{Du}+ C \kappa^l r_0 M_k(r_0) \int_0^{\kappa^l r_0} \frac{\tilde \omega_{\mathbf A}(t)}{t}\,dt.
\end{equation}

\subsection*{Estimate for $\vec p_j$}
We shall derive improved estimates for $\abs{\vec p_j}$ using the following lemma, where we set
\[
v(x):=u(x)-\frac{1}{2} \mathbf{S}_j x\cdot x-\vec p_j\cdot x
\]
so that
\begin{equation}		\label{eq1852mon}
Dv(x)=Du(x)-\mathbf{S}_j x-\vec p_j.
\end{equation}
\begin{lemma}			\label{lem1702sat}
For $0<r \le \frac12$, we have
\[
\sup_{B_{r}}\, \abs{D v} \le C \left\{\left( \fint_{B_{2r}} \abs{Dv}^{\frac12} \right)^{2} + (r \abs{\mathbf{S}_j}+\abs{\vec p_j}) \int_0^r \frac{\tilde \omega_{\mathbf A}(t)}{t}\,dt\right\},
\]
where $C=C(d, \lambda, \Lambda, \omega_{\mathbf A},\beta)$.
\end{lemma}

\begin{proof}
Since $\dv(\mathbf{A} Du)=0$ in $B_1$, it follows that
\[
\dv(\mathbf A Dv) = -\dv (\mathbf A(\mathbf{S}_j x+\vec p_j))  \quad\text{in}\quad B_{2r},
\]
for any $0<r \le \frac12$.
Let $x_0 \in B_{3r/2}$ and $0<t \le r/4$, and define $\bar{\mathbf A}:=(\mathbf A)_{B_t(x_0)}$.

We decompose $v$ as $v=v_1+v_2$, where $v_1 \in W^{1,p}_0(B_t(x_0))$ (for some $p>1$) is the weak solution to the problem
\[
\dv(\bar{\mathbf A} Dv_1) = -\dv ((\mathbf{A} -\bar{\mathbf A})(\mathbf{S}_jx+\vec p_j)+(\mathbf{A}- \bar{\mathbf A})Dv)\;\text{ in }\;B_t(x_0),
\]
with the boundary condition $v_1=0$ on $\partial B_t(x_0)$

By Lemma \ref{lem01}, we obtain (cf. \eqref{eq1622thu})
\[
\left(\fint_{B_t(x_0)} \abs{Dv_1}^{\frac12} \right)^{2}
\le C \left(\fint_{B_t(x_0)}\abs{\mathbf A-\bar{\mathbf A}}\right) \left(r\abs{\mathbf{S}_j}+\abs{\vec p_j}\right)
+C \left(\fint_{B_t(x_0)}\abs{\mathbf A-\bar{\mathbf A}}\right)\norm{Dv}_{L^\infty(B_t(x_0))},
\]
and thus, we have
\[
\left(\fint_{B_t(x_0)} \abs{Dv_1}^{\frac12} \right)^{2} \le C \omega_{\mathbf A}(t)\left(r\abs{\mathbf{S}_j}+\abs{\vec p_j}\right)+ C  \omega_{\mathbf A}(t)\norm{Dv}_{L^\infty(B_t(x_0))}.
\]

On the other hand, observe that $v_2=v-v_1$ satisfies
\[
L_0 v_2:=\dv(\bar{\mathbf A}Dv_2) = -\dv (\bar{\mathbf A}(\mathbf{S}_jx+\vec p_j)) = \text{constant} \quad\text{in }\,  B_t(x_0).
\]
Since $L_0$ is a constant-coefficients operator, it follows that $L_0(Dv_2)=0$.
The remainder of the proof then proceeds identically to that of \cite[Theorem 1.5]{DK17}.
\end{proof}

By Lemma \ref{lem1702sat}, and using \eqref{eq1852mon}, \eqref{eq0203tue}, and \eqref{eq4.38}, we have for $j=1,2,\ldots$:
\begin{align}
			\nonumber
\norm{Du-\mathbf{S}_j x-\vec p_j}_{L^\infty(B_{\frac12 \kappa^j r_0})} &\le C \kappa^j r_0 \varphi(\kappa^j r_0) +C \left(\kappa^j r_0\abs{\mathbf{S}_j}+\abs{\vec p_j}\right) \int_0^{\kappa^j r_0} \frac{\tilde\omega_{\mathbf A}(t)}{t}\,dt \\
			\nonumber
&\le C\kappa^{(1+\beta)j} \fint_{B_{r_0}} \abs{Du} + C\kappa^j r_0 M_j(r_0) \tilde\omega_{\mathbf A}(\kappa^j r_0)\\
				\label{eq0934wed0}
&\qquad+ C \left(\kappa^j r_0\abs{\mathbf{S}_j}+\abs{\vec p_j}\right) \int_0^{\kappa^j r_0} \frac{\tilde\omega_{\mathbf A}(t)}{t}\,dt,
\end{align}
where $C=C(d, \lambda, \Lambda, \omega_{\mathbf A}, \beta)$.
Since $Du(0)=0$, we infer from \eqref{eq0934wed0} that
\begin{multline*}
\abs{\vec p_j} \le C\kappa^{(1+\beta)j} \fint_{B_{r_0}} \abs{Du} + C\kappa^j r_0 M_j(r_0) \tilde\omega_{\mathbf A}(\kappa^j r_0)\\
+C \abs{\mathbf{S}_j}\kappa^j r_0 \int_0^{\kappa^j r_0} \frac{\tilde \omega_{\mathbf A}(t)}{t}\,dt
+C \abs{\vec p_j} \int_0^{\kappa^j r_0} \frac{\tilde\omega_{\mathbf A}(t)}{t}\,dt.
\end{multline*}
Let us fix $r_1>0$ such that
\begin{equation}			\label{eq0903tue}
C \int_0^{r_1} \frac{\tilde \omega_{\mathbf A}(t)}{t}\,dt\le \frac12.
\end{equation}
Note that that $r_1$ depends only on $d$, $\lambda$, $\Lambda$, $\omega_{\mathbf A}$, and $\beta$.
We have not yet chosen $r_0 \in (0,\frac12]$.
We will require $r_0 \le r_1$, which implies
\[
\abs{\vec p_j} \le  C\kappa^{(1+\beta)j} \fint_{B_{r_0}} \abs{Du} + C\kappa^j r_0 M_j(r_0) \tilde\omega_{\mathbf A}(\kappa^j r_0)
+C\abs{\mathbf{S}_j}\kappa^j r_0\int_0^{\kappa^j r_0} \frac{\tilde \omega_{\mathbf A}(t)}{t}\,dt.
\]
This estimate, combined with \eqref{eq0900tue} and \eqref{rmk1147_2}, yields
\begin{equation}			\label{eq7.51}
\abs{\vec p_j} \le C \kappa^{j} r_0 \left\{\kappa^{\beta j}+\int_0^{\kappa^j r_0} \frac{\tilde \omega_{\mathbf A}(t)}{t}\,dt\right\}\frac{1}{r_0} \fint_{B_{r_0}} \abs{Du} +C\kappa^j r_0 M_j(r_0) \int_0^{\kappa^j r_0} \frac{\tilde \omega_{\mathbf A}(t)}{t}\,dt.
\end{equation}

\subsection*{Convergence of $\mathbf{S}_j$}
By \eqref{eq0934wed0}, \eqref{eq0900tue}, \eqref{eq7.51}, and \eqref{rmk1147_2}, we have
\begin{multline}				\label{eq1920thu}
\bignorm{Du-\mathbf{S}_jx-\vec p_j}_{L^\infty(B_{\frac12 \kappa^j r_0})}
\le C \kappa^j r_0 \left\{\kappa^{\beta j}+ \int_0^{\kappa^j r_0} \frac{\tilde \omega_{\mathbf A}(t)}{t}\,dt \right\}\frac{1}{r_0} \fint_{B_{r_0}} \abs{Du}\\
+C\kappa^j r_0 M_j(r_0)  \int_0^{\kappa^j r_0} \frac{\tilde \omega_{\mathbf A}(t)}{t}\,dt.
\end{multline}
Then, from \eqref{eq1920thu}, \eqref{eq0900tue}, \eqref{eq7.51}, and \eqref{rmk1147_2}, we infer that
\begin{equation}				\label{eq1555sun}
\frac {1}{\kappa^{j} r_0}\norm{Du}_{L^\infty(B_{\frac12 \kappa^j r_0})}
\le \frac C {r_0} \fint_{B_{r_0}} \abs{Du} + C M_j(r_0) \int_0^{r_0} \frac{\tilde \omega_{\mathbf A}(t)}{t}\,dt,
\end{equation}
where $C=C(d, \lambda, \Lambda, \omega_{\mathbf A}, \beta)$.

\begin{lemma}				\label{lem1548sun}
There exists a constant $r_0=r_0(d, \lambda, \Lambda, \omega_{\mathbf A}, \beta) \in (0,\frac12)$ such that
\begin{equation}			\label{eq1951thu}
\sup_{j \ge 1} M_j(r_0) =\sup_{i \ge 0} \frac{1}{\kappa^i r_0} \,\norm{Du}_{L^\infty(B_{\kappa^i r_0})} \le \frac{C}{r_0} \fint_{B_{2r_0}} \abs{Du},
\end{equation}
where $C=C(d, \lambda, \Lambda, \omega_{\mathbf A}, \beta)$.
\end{lemma}

\begin{proof}
We define
\[
c_i= \frac{1}{\kappa^i r_0} \norm{Du}_{L^\infty(B_{\kappa^i r_0})},\quad i=0,1,2,\ldots.
\]
Then, by \eqref{eq1244sat}, it is clear that
\begin{equation}			\label{eq0906tue}
M_1(r_0)=c_0\quad\text{and}\quad
M_{j+1}(r_0)=\max(M_j(r_0), c_j),\quad j=1,2,\ldots.
\end{equation}
Recall that $\kappa = \kappa(d, \lambda, \Lambda, \beta) \in (0, \tfrac{1}{2})$ has already been chosen. Then, from \eqref{eq1555sun}, we deduce that there exists a constant $C=C(d, \lambda, \Lambda, \omega_{\mathbf A}, \beta)>0$ such that
\begin{equation}			\label{eq0848tue}
c_{j+1} \le C \left\{\frac{1}{r_0} \fint_{B_{r_0}}\abs{Du}+ M_j(r_0) \int_0^{r_0} \frac{\tilde \omega_{\mathbf A}(t)}{t}\,dt \right\},\quad j=1,2,\ldots.
\end{equation}
By \cite[Theorem 1.5]{DK17}, we obtain
\begin{equation}			\label{eq0849tue}
c_1 =\frac{1}{\kappa r_0} \norm{Du}_{L^\infty(B_{\kappa r_0})} \le \frac{1}{\kappa r_0} \norm{Du}_{L^\infty(B_{r_0})} =\frac{1}{\kappa} c_0  \le \frac{C}{r_0} \fint_{B_{2r_0}} \abs{Du},
\end{equation}
where $C=C(d, \lambda, \Lambda, \omega_{\mathbf A},\beta)$.

By combining \eqref{eq0848tue} and \eqref{eq0849tue}, we establish the existence of a constant $\gamma=\gamma(d, \lambda, \Lambda,\omega_{\mathbf A}, \beta)>0$ such that the following inequalities hold:
\begin{align*}
c_0,\, c_1 &\le \frac{\gamma}{r_0} \fint_{B_{2r_0}}\abs{Du}\quad\text{and}\\
c_{j+1} &\le \gamma \left\{\frac{1}{r_0} \fint_{B_{2r_0}}\abs{Du}+ M_j(r_0) \int_0^{r_0} \frac{\tilde \omega_{\mathbf A}(t)}{t}\,dt \right\},\;\; j=1,2,\ldots.
\end{align*}

Now, we fix a number $r_0 \in (0,\frac12]$ such that
\[
\gamma \int_0^{r_0} \frac{\tilde \omega_{\mathbf A}(t)}{t}\,dt \le \frac12
\]
and also ensure that $r_0 \le r_1$, is as defined in \eqref{eq0903tue}.
With this choice, we obtain
\begin{equation}			\label{eq2114mon}
c_0, c_1 \le \frac{\gamma}{r_0} \fint_{B_{2r_0}}\abs{Du} \quad \text{and}\quad c_{j+1} \le \frac{\gamma}{r_0} \fint_{B_{2r_0}}\abs{Du} +\frac12 M_j(r_0),\;\; j=1,2,\ldots.
\end{equation}

By induction, it follows form \eqref{eq0906tue} and \eqref{eq2114mon} that
\[
c_{2k},\,c_{2k+1},\,M_{2k+1}(r_0),\,M_{2k+2}(r_0) \le \frac{\gamma}{r_0} \fint_{B_{2r_0}}\abs{Du}\cdot \sum_{i=0}^{k} \frac{1}{2^i},\quad k=0,1,2,\ldots.
\]
This completes the proof of the lemma.
\end{proof}

\smallskip

Now, Lemma \ref{lem1548sun} and \eqref{eq0946tue} imply that the sequence $\{\mathbf{S}_j\}$ is a Cauchy sequence in $\mathbb{S}^d$, and thus $\mathbf{S}_j \to \mathbf{S}$ for some $\mathbf{S} \in \mathbb{S}^d$.
Moreover, by taking the limit as $k\to \infty$ in \eqref{eq0946tue} and \eqref{eq0947tue} (while recalling \eqref{eq1951thu} and \eqref{eq0215tue}), respectively, and then setting $l=j$, we obtain the following estimates:
\begin{equation}	\label{eq1806sun}
\begin{aligned}
\abs{\mathbf{S}_j-\mathbf{S}} &\le C \left\{ \kappa^{\beta j}+ \int_0^{\kappa^j r_0} \frac{\tilde \omega_{\mathbf A}(t)}{t}\,dt\right\}\frac{1}{r_0} \fint_{B_{2r_0}} \abs{Du},\\
\abs{\vec p_j}  &\le C \kappa^j \left\{\kappa^{\beta j} +\int_0^{\kappa^j r_0} \frac{\tilde \omega_{\mathbf A}(t)}{t}\,dt \right\}\fint_{B_{2r_0}} \abs{Du}.
\end{aligned}
\end{equation}

By the triangle inequality, \eqref{eq1920thu}, \eqref{eq1806sun}, and \eqref{eq1951thu}, we obtain
\begin{align}			\nonumber
\norm{Du-\mathbf{S}x}_{L^\infty(B_{\frac12 \kappa^j r_0})}& \le \norm{Du- \mathbf{S}_j x -\vec  p_j}_{L^\infty(B_{\frac12 \kappa^j r_0})} + \frac{\kappa^j r_0}{2} \abs{\mathbf{S}_j - \mathbf{S}} +  \abs{\vec p_j}\\
					\label{eq2221sun}
&\le C \kappa^j r_0\left\{\kappa^{\beta j}+\int_0^{\kappa^j r_0} \frac{\tilde\omega_{\mathbf A}(t)}{t}\,dt\right\} \frac{1}{r_0} \fint_{B_{2r_0}} \abs{Du}.
\end{align}

\subsection*{Conclusion}
It follows from \eqref{eq2221sun} that
\begin{equation}		\label{eq1036wed}
\frac{1}{r} \norm{Du-\mathbf{S}x}_{L^\infty(B_r)}  \le \varrho_{\mathbf A}(r) \left( \frac{1}{r_0} \fint_{B_{2r_0}} \abs{Du} \right),
\end{equation}
where
\begin{equation}			\label{eq2143sat}
\varrho_{\mathbf A}(r)=C\left\{\left(\frac{2r}{\kappa r_0}\right)^\beta+\int_0^{2r/\kappa} \frac{\tilde\omega_{\mathbf A}(t)}{t}\,dt\right\}.
\end{equation}
Note that $\varrho_{\mathbf A}$ is a modulus of continuity  determined by $d$, $\lambda$, $\Lambda$, $\omega_{\mathbf A}$, and $\beta \in (0,1)$.

In particular, we conclude from \eqref{eq1036wed} that $Du$ is differentiable at $0$.
Moreover, it follows from \eqref{eq0900tue} and \eqref{eq1951thu} that (noting that $D^2u(0)=\mathbf{S}=\lim_{j\to \infty} \mathbf{S}_j$) we have
\[
\abs{D^2 u (0)} \le \frac{C}{r_0} \fint_{B_{2r_0}} \abs{Du},
\]
where $C=C(d, \lambda, \Lambda, \omega_{\mathbf A},\beta)$.

From \eqref{eq2143sat} and \eqref{eq1245sat}, we observe that when $\mathbf A \in C^{\alpha}$ for some $\alpha \in (0,1)$, we have $\varrho_{\mathbf A}(r) \lesssim r^\alpha$ by choosing $\beta \in (\alpha,1)$.
This completes the proof in the special case.
 \qed

\section{Proof of Theorem \ref{thm01}: General case}		\label{sec3}
We now proceed with the proof of Theorem~\ref{thm01} in the general setting.

\smallskip
Define $\omega_{\rm coef}(\cdot)$ by
\begin{equation} \label{omega_coef_div}
\omega_{\rm coef}(r) := \omega_{\mathbf A}(r) + r^{1-d}\sup_{x \in \Omega} \int_{\Omega \cap B_r(x)}\abs{\vec b} +r^{1-d} \sup_{x \in \Omega} \int_{\Omega \cap B_r(x)}\abs{c}.
\end{equation}
We observe that $\omega_{\rm coef}(\cdot)$ satisfies the Dini condition:
\[
\int_0^1 \frac{\omega_{\rm coef}(t)}{t} \,dt <\infty.
\]
Although one might expect a factor of $r^{2-d}$ in the last term of the definition \eqref{omega_coef_div}, we retain the current form for consistency with the notation $\omega_{\rm coef}(\cdot,x_0)$ introduced in the proof of  Lemma~\ref{lem3.11sat}.

Let $u$ be a solution to \eqref{eq_main-d}, i.e., $u$ satisfies
\[
\dv (\mathbf{A} Du) + \boldsymbol{b} \cdot Du + cu = f \quad \text{in }\; \Omega.
\]
From \cite[Theorem 1.3]{DEK18} (and also \cite[Proposition 2.6]{DEK18}), it follows that $u \in C^1_{\rm loc}(\Omega)$.
Assume that $Du(x^o) = 0$ for some $x^o \in \Omega$.
Without loss of generality, we can set $x^o = 0$ and assume $B_2(0) \subset \Omega$.

For $0 < r \le 1/2$, we define $\bar{\mathbf A}$, $\bar c$, $\bar u$, and $\bar f$ as the averages of $\mathbf A$, $c$, $u$, and $f$ over the ball $B_r$, respectively.

We decompose $u$ as  $u=v+w$, where $w \in W^{1,p}_0(B_r)$ (for some $p>1$) is the solution to the problem:
\[
\dv( \bar{\mathbf A} Dw)= -\dv ((\mathbf{A}-\bar{\mathbf A})Du)-\vec b \cdot Du-(cu -\bar c \bar u)+f-\bar f\;\mbox{ in }\; B_r,
\]
with $w=0$ on $\partial B_r$.

By Lemma~\ref{lem01}, we obtain the following estimate via rescaling (cf. \eqref{eq1622thu}):
\begin{multline}			\label{eq1020fri}
\left(\fint_{B_r} \abs{Dw}^{\frac12}\right)^{2} \le C\omega_{\mathbf A}(r) \norm{Du}_{L^\infty(B_r)}+Cr\left(\fint_{B_r} \abs{\vec b}\right) \norm{Du}_{L^\infty(B_r)}\\
 +Cr\left\{\omega_{c}(r)\norm{u}_{L^\infty(B_r)}+r\left(\fint_{B_r} \abs{c} \right)\norm{Du}_{L^\infty(B_r)}\right\}+ Cr \omega_f(r),
\end{multline}
where we utilized the identity
\[
cu-\bar c \bar u=(c-\bar c)u+\bar c(u-\bar u)
\]
and applied the Poincar\'e inequality to derive:
\begin{align}
				\nonumber
\fint_{B_r} \abs{cu-\bar c \bar u} &\le \norm{u}_{L^\infty(B_r)} \fint_{B_r} \abs{c-\bar c}+ C\abs{\bar c} r \fint_{B_r} \abs{Du}\\
				\label{eq1042fri}
& \le \omega_{c}(r)\norm{u}_{L^\infty(B_r)}+ Cr \left(\fint_{B_r} \abs{c}\right)\norm{Du}_{L^\infty(B_r)}.
\end{align}

Since $v=u-w$ satisfies
\[
\dv(\bar{\mathbf{A}} Dv) = \bar{f} - \bar{c} \bar{u} = \text{constant} \quad \text{in } B_r,
\]
the function $\tilde v:=D_kv - l$, where $k=1,\ldots, d$ and $l(x)$ is any affine function, satisfies
\[
\dv(\bar{\mathbf{A}} D \tilde v)=0\quad \text{in } B_r.
\]
Therefore, the same reasoning that led to estimate \eqref{eq1228thu} also applies here, yielding the identical estimate.

Let $\varphi(r)$ be as defined in \eqref{eq1100sat}, and let $\beta \in (0,1)$ be  an arbitrary but fixed constant.

By using \eqref{eq1020fri} in place of \eqref{eq1622thu} and carrying out computations analogous to those in \eqref{eq0224sat}, we arrive at the following estimate, which is analogous to \eqref{eq1110sat}:
\begin{equation}		\label{eq2154thu}
\varphi(\kappa r) \le \kappa^\beta \varphi(r) + C \omega_{\rm coef}(r)\, \frac{1}{r} \norm{Du}_{L^\infty(B_r)} +C \omega_f(r)+ C \omega_c(r) \norm{u}_{L^\infty(B_r)}
\end{equation}
where $C=C(d, \lambda, \Lambda, \beta)$.
Recall that for a given $\beta$, the constant $\kappa=\kappa(d,\lambda, \Lambda, \beta)$ is chosen to satisfy $\kappa \in (0, 1/2)$.

Let $M_j(r_0)$ be defined as in \eqref{eq1244sat}, where $r_0\in (0, \frac{1}{2}]$ is a parameter to be chosen later.
Define $\tilde \omega_{\cdots}(\cdot)$ as in \eqref{eq1245sat}, replacing $\omega_{\mathbf A}(\cdot)$ with $\omega_{\cdots}(\cdot)$.

Then, by replicating the argument from Section~\ref{sec2}, we obtain the following estimate, analogous to \eqref{eq4.38}:
\begin{equation}			\label{eq0616thu}
\varphi(\kappa^j r_0) \le  \frac{\kappa^{\beta j}}{r_0} \fint_{B_{r_0}} \abs{Du}+ CM_j(r_0) \tilde \omega_{\rm coef}(\kappa^j r_0) +C \tilde \omega_f(\kappa^j r_0) + C \tilde \omega_c(\kappa^j r_0)\norm{u}_{L^\infty(B_{r_0})}.
\end{equation}

For $j=0,1,2,\ldots$, let $\mathbf{S}_j  \in \mathbb{S}^d$ and $\vec p_j \in \bR^d$ be chosen as in \eqref{eq0203tue}.
Then, by using \eqref{eq0616thu} in place of \eqref{eq4.38} and replicating the argument from Section~\ref{sec2}, we arrive at the following conclusion:
\begin{equation}			\label{eq1717thu}
\lim_{j\to \infty} \vec p_j=0.
\end{equation}

Additionally, by repeating the same computations that lead to \eqref{eq0946tue} and \eqref{eq0947tue}, we obtain the following estimate for $k>l \ge 0$:
\begin{multline}			\label{eq1718thu}
\abs{\mathbf{S}_k -\mathbf{S}_l} +\frac{\abs{\vec p_k - \vec p_l}}{\kappa^l r_0}
\le \frac{C \kappa^{\beta l}}{r_0} \fint_{B_{r_0}} \abs{Du}+ CM_k(r_0) \int_0^{\kappa^l r_0} \frac{\tilde \omega_{\rm coef}(t)}{t}\,dt\\
+C \int_0^{\kappa^l r_0} \frac{\tilde \omega_f(t)}{t}\,dt
+C\norm{u}_{L^\infty(B_{r_0})} \int_0^{\kappa^l r_0} \frac{\tilde \omega_c(t)}{t}\,dt,
\end{multline}
where $C=C(d, \lambda, \Lambda,\beta)$.
In particular, similar to \eqref{eq0900tue}, we obtain the following bound for $\mathbf{S}_j$:
\begin{multline}			\label{eq1628fri}
\abs{\mathbf S_j} \le  \frac{C}{r_0} \fint_{B_{r_0}} \abs{Du}+ CM_j(r_0) \int_0^{r_0} \frac{\tilde \omega_{\rm coef}(t)}{t}\,dt\\
+C \int_0^{r_0} \frac{\tilde \omega_f(t)}{t}\,dt + C\norm{u}_{L^\infty(B_{r_0})} \int_0^{r_0} \frac{\tilde \omega_c(t)}{t}\,dt,
\end{multline}
where $C=C(d, \lambda, \Lambda, \beta)$.

\medskip
The following lemma corresponds to Lemma \ref{lem1702sat}.

\begin{lemma}			\label{lem3.11sat}
Let $v$ be defined by
\begin{equation}			\label{eq1359fri}
v(x):=u(x)-\tfrac{1}{2} \mathbf{S}_j x\cdot x-\vec p_j\cdot x.
\end{equation}
Then, for $0<r \le \frac12$, we have
\begin{multline*}
\sup_{B_{r}}\, \abs{Dv} \le C \left\{ \left( \fint_{B_{2r}} \abs{Dv}^{\frac12} \right)^{2} + r\int_0^r \frac{\tilde \omega_f(t)}{t}\,dt +r\norm{u}_{L^\infty(B_r)}\int_0^r \frac{\tilde \omega_c(t)}{t}\,dt\right.\\
\left.+(r\abs{\mathbf S_j}+\abs{\vec p_j}) \int_0^r \frac{\tilde \omega_{\rm coef}(t)}{t}\,dt \right\},
\end{multline*}
where $C=C(d, \lambda, \Lambda, \omega_{\rm coef},\beta)$.
\end{lemma}
\begin{proof}
Note that $Dv$ is given by
\begin{equation}		\label{eq0846sun}
Dv(x)=Du(x)-\mathbf{S}_j x-\vec p_j.
\end{equation}
For $x_0 \in B_{3r/2}$ and $0<t\le r/4$, let $\bar{\mathbf A}$, $\bar c$, $\bar v$, and $\bar{f}$ denote the averages of $\mathbf A$, $c$, $v$, and $f$ over the ball $B_t(x_0)$, respectively.
Observe that $v$ satisfies
\begin{multline*}
\dv(\bar{\mathbf A}Dv)=f -\dv((\mathbf{A}-\bar{\mathbf A})Dv)-\vec b\cdot Dv-cv\\
-\dv(\mathbf A (\mathbf{S}_j  x + \vec p_j)) - \vec b\cdot (\mathbf S_j x + \vec p_j) -c (\tfrac{1}{2}\mathbf  S_j x \cdot x +\vec  p_j \cdot x)\quad\mbox{in }\, B_t(x_0).
\end{multline*}

We decompose $v$ as $v=v_1+v_2$, where $v_1 \in W^{1,p}_0(B_t(x_0))$ for some $p>1$,  and $v_1$ solves
\begin{align*}
\dv(\bar{\mathbf A}Dv_1) &=  f-\bar f -\dv((\mathbf{A}- \bar{\mathbf A})Dv)-\vec b\cdot Dv-(cv-\bar c \bar v)\\
&-\dv((\mathbf A-\bar{\mathbf A}) (\mathbf{S}_j  x + \vec p_j)) - \vec b\cdot (\mathbf S_j x + \vec p_j) -c (\tfrac{1}{2}\mathbf  S_j x \cdot x +\vec  p_j \cdot x)\quad\mbox{in }\, B_t(x_0)
\end{align*}
with boundary condition $v_1=0$ on $\partial B_t(x_0)$.

Then, applying Lemma~\ref{lem01} via rescaling, we obtain the following estimate:
\begin{align}
			\nonumber
\left(\fint_{B_t(x_0)} \abs{Dv_1}^{\frac12} \right)^{2} &\lesssim t \fint_{B_t(x_0)}\abs{f-\bar f}+ \norm{Dv}_{L^\infty(B_t(x_0))}\fint_{B_t(x_0)}\abs{\mathbf A-\bar{\mathbf A}}+t \norm{Dv}_{L^\infty(B_t(x_0))} \fint_{B_t(x_0)}\abs{\vec b} \\
			\nonumber
&+t \fint_{B_t(x_0)} \abs{cv-\bar c \bar v} +(r\abs{\mathbf S_j}+\abs{\vec p_j})\fint_{B_t(x_0)}\abs{\mathbf A-\bar{\mathbf A}}\\
				\label{eq1116mon}
&+ t(r\abs{\mathbf S_j}+\abs{\vec p_j})\fint_{B_t(x_0)}\abs{\vec b}+tr(r\abs{\mathbf S_j}+\abs{\vec p_j})\fint_{B_t(x_0)}\abs{c}.
\end{align}
Similar to estimate \eqref{eq1042fri}, we have:
\begin{align}
			\nonumber
&\fint_{B_t(x_0)} \abs{cv-\bar c \bar v} \le \norm{v}_{L^\infty(B_t(x_0))}\fint_{B_t(x_0)} \abs{c-\bar c} + Ct\norm{Dv}_{L^\infty(B_t(x_0))}\fint_{B_t(x_0)} \abs{c} \\
			\nonumber
&\qquad\le \left\{\norm{u}_{L^\infty(B_t(x_0))}+ 2r^2\abs{\mathbf S_j}+2r\abs{\vec p_j}\right\}\fint_{B_t(x_0)} \abs{c-\bar c}+Ct\norm{Dv}_{L^\infty(B_t(x_0))}\fint_{B_t(x_0)} \abs{c},\\
			\label{eq1922sat}
&\qquad \le \omega_c(t) \norm{u}_{L^\infty(B_t(x_0))}+ 4r(r\abs{\mathbf S_j}+\abs{\vec p_j})\fint_{B_t(x_0)} \abs{c} +Ct\norm{Dv}_{L^\infty(B_t(x_0))}\fint_{B_t(x_0)} \abs{c},
\end{align}
where we used \eqref{eq1359fri} in the second inequality.

Thus, by the definition \eqref{omega_coef_div}, and using \eqref{eq1116mon} and \eqref{eq1922sat}, we obtain
\begin{multline*}
\left(\fint_{B_t(x_0)} \abs{Dv_1}^{\frac12} \right)^{2} \le Ct\omega_f(t)+C \omega_{\rm coef}(t)\norm{Dv}_{L^\infty(B_t(x_0))} + Ct\omega_c(t)\norm{u}_{L^\infty(B_r)}\\
+C(r\abs{\mathbf S_j}+\abs{\vec p_j}) \omega_{\rm coef}(t,x_0),
\end{multline*}
where $C=C(d, \lambda, \Lambda)$, and we define (noting that $r\le 1)$.
\[
\omega_{\rm coef}(t,x_0):=\omega_{\mathbf A}(t)+t\fint_{B_t(x_0)}\abs{\vec b} + t\fint_{B_t(x_0)}\abs{c}.
\]

On the other hand, note that $v_2=v-v_1$ satisfies
\[
L_0v_2:=\dv(\bar{\mathbf A}Dv_2) = \mbox{constant} \quad\text{ in } B_t(x_0).
\]
Therefore, $L_0(D v_2)=0$ in $B_t(x_0)$.
The remainder of the proof proceeds by arguments analogous to those in \cite[Lemma 2.2]{KL21} and \cite[Theorem 1.3]{DEK18}.
\end{proof}

By applying Lemma~\ref{lem3.11sat} together with \eqref{eq0616thu}, \eqref{eq0846sun}, and \eqref{rmk1147_2}, we derive the following estimate (cf. \eqref{eq0934wed0}):
\begin{align}
			\nonumber
&\norm{Du-\mathbf{S}_j x-\vec p_j}_{L^\infty(B_{\frac12 \kappa^j r_0})}\le C\kappa^{(1+\beta)j} \fint_{B_{r_0}} \abs{Du}+C\kappa^j r_0 M_j(r_0) \tilde\omega_{\rm coef}(\kappa^j r_0)\\
			\nonumber
&\quad\qquad
+C \left(\abs{\mathbf S_j}\kappa^j r_0+\abs{\vec p_j}\right) \int_0^{\kappa^j r_0} \frac{\tilde \omega_{\rm coef}(t)}{t}\,dt + C \kappa^j r_0 \int_0^{\kappa^j r_0} \frac{\tilde \omega_f(t)}{t}\,dt\\
				\label{eq1218fri}
&\quad\qquad+ C \kappa^j r_0\norm{u}_{L^\infty(B_{r_0})} \int_0^{\kappa^j r_0} \frac{\tilde \omega_c(t)}{t}\,dt.
\end{align}
Since $Du(0)=0$, we infer from \eqref{eq1218fri} that
\begin{align}
			\nonumber
\abs{\vec p_j} &\le C\kappa^{(1+\beta)j} \fint_{B_{r_0}} \abs{Du}+C\kappa^j r_0 M_j(r_0) \tilde\omega_{\rm coef}(\kappa^j r_0)\\ 
			\nonumber
&\quad +C \kappa^j r_0 \abs{\mathbf S_j} \int_0^{\kappa^j r_0} \frac{\tilde \omega_{\rm coef}(t)}{t}\,dt
+C \abs{\vec p_j} \int_0^{\kappa^j r_0} \frac{\tilde \omega_{\rm coef}(t)}{t}\,dt \\
				\label{eq0911sun}
&\quad+ C \kappa^j r_0 \int_0^{\kappa^j r_0} \frac{\tilde \omega_f(t)}{t}\,dt+ C \kappa^j r_0\norm{u}_{L^\infty(B_{r_0})} \int_0^{\kappa^j r_0} \frac{\tilde \omega_c(t)}{t}\,dt.
\end{align}

We require $r_0 \in (0,\frac12]$ to satisfy $r_0 \le r_1$, where $r_1= r_1(d, \lambda, \Lambda, \omega_{\rm coef}, \beta)>0$ is chosen so that
\[
C \int_0^{r_1} \frac{\tilde \omega_{\rm coef}(t)}{t}\,dt\le \frac12.
\]
Similarly to \eqref{eq7.51}, we derive from \eqref{eq0911sun} and \eqref{eq1628fri} the following inequality:
\begin{multline}			\label{eq1123fri}
\abs{\vec p_j} \le C \kappa^{j}r_0 \left\{\kappa^{\beta j}+\int_0^{\kappa^j r_0} \frac{\tilde \omega_{\rm coef}(t)}{t}\,dt\right\}\frac{1}{r_0} \fint_{B_{r_0}} \abs{Du}+C\kappa^j r_0 M_j(r_0)\int_0^{\kappa^j r_0}\frac{\tilde\omega_{\rm coef}(t)}{t}\,dt\\
+ C \kappa^j r_0 \int_0^{r_0}\frac{\tilde\omega_f(t)}{t}\,dt +C \kappa^j r_0 \norm{u}_{L^\infty(B_{r_0})} \int_0^{r_0}\frac{\tilde\omega_c(t)}{t}\,dt,
\end{multline}
where $C=C(d, \lambda, \Lambda,\omega_{\rm coef},\beta)$.

Then, by applying \eqref{eq1218fri}, \eqref{eq1628fri}, and \eqref{eq1123fri}, we obtain the following estimate, which is analogous to \eqref{eq1920thu}:
\begin{align}				\nonumber
&\norm{Du-\mathbf S_jx-\vec p_j}_{L^\infty(B_{\frac12 \kappa^j r_0})}  \le C \kappa^j r_0 \left\{\kappa^{\beta j}+ \int_0^{\kappa^j r_0} \frac{\tilde \omega_{\rm coef}(t)}{t}\,dt \right\} \frac{1}{r_0} \fint_{B_{r_0}} \abs{Du}\\
						\nonumber
&\quad+C \kappa^j r_0 M_j(r_0) \int_0^{\kappa^j r_0} \frac{\tilde\omega_{\rm coef}(t)}{t}\,dt + C \kappa^j r_0 \left( \int_0^{\kappa^j r_0} \frac{\tilde \omega_f(t)}{t}\,dt + \norm{u}_{L^\infty(B_{r_0})}\int_0^{\kappa^j r_0} \frac{\tilde \omega_c(t)}{t}\,dt\right)\\
						\label{eq1720fri}
&\quad +C \kappa^j r_0 \left(\int_0^{r_0} \frac{\tilde \omega_f(t)}{t}\,dt+\norm{u}_{L^\infty(B_{r_0})}\int_0^{r_0} \frac{\tilde \omega_c(t)}{t}\,dt\right)\int_0^{\kappa^j r_0} \frac{\tilde \omega_{\rm coef}(t)}{t}\,dt.
\end{align}
Additionally, similar to \eqref{eq1555sun}, we also obtain
\begin{multline}				\label{eq1244fri}
\frac 2 {\kappa^j r_0}\norm{Du}_{L^\infty(B_{\frac12 \kappa^j r_0})}
\le \frac C {r_0} \fint_{B_{r_0}} \abs{Du} + C M_j(r_0) \int_0^{r_0} \frac{\tilde \omega_{\rm coef}(t)}{t}\,dt\\
+C \int_0^{r_0}\frac{\tilde\omega_f(t)}{t}\,dt+ C \norm{u}_{L^\infty(B_{r_0})}\int_0^{r_0}\frac{\tilde\omega_c(t)}{t}\,dt,
\end{multline}
where $C=C(d, \lambda, \Lambda, \omega_{\rm coef}, \beta)$.

Then, by following the same proof as in Lemma~\ref{lem1548sun}, we conclude from \eqref{eq1244fri} that there exists $r_0=r_0(d, \lambda, \Lambda, \omega_{\rm coef}, \beta) \in (0,\frac12]$ such that
\begin{equation}			\label{eq1706fri}
\sup_{j \ge 1} M_j(r_0) \le \frac{C}{r_0} \fint_{B_{2r_0}} \abs{Du}+C \int_0^{r_0} \frac{\tilde \omega_f(t)}{t}\,dt + C \norm{u}_{L^\infty(B_{r_0})} \int_0^{r_0} \frac{\tilde \omega_c(t)}{t}\,dt,
\end{equation}
where $C=C(d, \lambda, \Lambda, \omega_{\rm coef}, \beta)$.

It follows from \eqref{eq1718thu} and \eqref{eq1706fri} that the sequence $\{\mathbf S_j\}$ converges to some $\mathbf S \in \mathbb{S}^d$, as it is a Cauchy sequence in $\mathbb{S}^d$.
Taking the limit as $k\to \infty$ in \eqref{eq1718thu} (with reference to \eqref{eq1706fri} and \eqref{eq1717thu}), and subsequently setting $l=j$, yields the following estimate:
\begin{align}
				\nonumber
\abs{\mathbf S_j -\mathbf S} + \frac{\abs{\vec p_j}}{\kappa^j r_0} &\le C \left( \kappa^{\beta j}+\int_0^{\kappa^j r_0} \frac{\tilde \omega_{\rm coef}(t)}{t}\,dt \right) \frac{1}{r_0} \fint_{B_{2r_0}} \abs{Du}\\
				\nonumber
&+C \left(\int_0^{r_0} \frac{\tilde \omega_f(t)}{t}\,dt +\norm{u}_{L^\infty(B_{r_0})}\int_0^{r_0} \frac{\tilde \omega_c(t)}{t}\,dt \right) \int_0^{\kappa^j r_0} \frac{\tilde \omega_{\rm coef}(t)}{t}\,dt\\
			\label{eq2203fri}
&+ C \int_0^{\kappa^j r_0} \frac{\tilde \omega_f(t)}{t}\,dt + C \norm{u}_{L^\infty(B_{r_0})}\int_0^{\kappa^j r_0} \frac{\tilde \omega_c(t)}{t}\,dt,
\end{align}
where $C=C(d, \lambda, \Lambda,\omega_{\rm coef}, \beta)$.

Then, similar to \eqref{eq2221sun}, it follows from \eqref{eq1720fri}, \eqref{eq1706fri}, and \eqref{eq2203fri} that
\begin{align*}
\norm{Du-\mathbf{S}x}_{L^\infty(B_{\frac12 \kappa^j r_0})} &\le
C \kappa^j r_0 \left\{\kappa^{\beta j}+ \int_0^{\kappa^j r_0} \frac{\tilde \omega_{\rm coef}(t)}{t}\,dt \right\}\frac{1}{r_0} \fint_{B_{2r_0}} \abs{Du}\\
&+C \kappa^j r_0 \left\{\int_0^{r_0} \frac{\tilde \omega_f(t)}{t}\,dt+\norm{u}_{L^\infty(B_{r_0})}\int_0^{r_0} \frac{\tilde \omega_c(t)}{t}\,dt\right\} \int_0^{\kappa^j r_0} \frac{\tilde\omega_{\rm coef}(t)}{t}\,dt\\
&+ C \kappa^j r_0\left\{ \int_0^{\kappa^j r_0} \frac{\tilde \omega_f(t)}{t}\,dt +\norm{u}_{L^\infty(B_{r_0})}\int_0^{\kappa^j r_0} \frac{\tilde \omega_c(t)}{t}\,dt \right\}.
\end{align*}

The previous inequality yields the following estimate:
\begin{multline}		\label{eq2010fri}
\frac{1}{r} \norm{Du-\mathbf{S} x}_{L^\infty(B_r)}
\le \varrho_{\rm coef}(r) \left\{\frac{1}{r_0} \fint_{B_{2r_0}} \abs{Du} + \varrho_c(r_0) \norm{u}_{L^\infty(B_{r_0})} + \varrho_{f}(r_0)\right\}\\
+\varrho_c(r)\norm{u}_{L^\infty(B_{r_0})}+\varrho_f(r)\quad \text{for all } r \in (0, r_0/2).
\end{multline}
Here, the moduli of continuity $\varrho_{\rm coef}(\cdot)$, $\varrho_c(\cdot)$, and $\varrho_f(\cdot)$ are defined as:
\begin{equation*}		
\begin{aligned}
\varrho_{\rm coef}(r)&:=C\left\{\left(\frac{2r}{\kappa r_0}\right)^\beta+\int_0^{2r/\kappa}\frac{\tilde\omega_{\rm coef}(t)}{t}\,dt\right\},\\
\varrho_c(r)&:=C \int_0^{2r/\kappa} \frac{\tilde \omega_c(t)}{t}\,dt,\\
\varrho_f(r)&:=C \int_0^{2r/\kappa} \frac{\tilde \omega_f(t)}{t}\,dt,
\end{aligned}
\end{equation*}
where $C=C(d, \lambda, \Lambda,\omega_{\rm coef}, \beta)$.

In particular, it follows from \eqref{eq2010fri} that $Du$ is differentiable at $0$.
Moreover, by combining \eqref{eq1628fri} and \eqref{eq1706fri}, and noting that $D^2u(0)=\mathbf{S}=\lim_{j\to \infty} \mathbf S_j$, we obtain
\[
\abs{D^2 u (0)} \le C \left\{\frac{1}{r_0} \fint_{B_{2r_0}} \abs{Du}+ \int_0^{r_0} \frac{\tilde \omega_f(t)}{t}\,dt+ \norm{u}_{L^\infty(B_{r_0})} \int_0^{r_0} \frac{\tilde \omega_c(t)}{t}\,dt\right\},
\]
where $C=C(d, \lambda, \Lambda, \omega_{\rm coef}, \beta)$.
Here we define
\[
\varrho_{\rm lot}(r):=\norm{u}_{L^\infty(B_{r_0})} \int_0^{r} \frac{\tilde \omega_c(t)}{t}\,dt.
\]
Note that $\varrho_{\rm lot}(r) \lesssim r^\alpha$ if $c \in C^\alpha$, and $\varrho_{\rm lot} \equiv 0$ if $c$ is constant.
This completes the proof of the theorem.
\qed

\section{Proof of Theorem \ref{thm02}: Simple case}			\label{sec4}
In this section, we consider an elliptic operator $\mathrm{L}$ of the form
\[
\mathrm{L} u=  a^{ij} D_{ij}u=\tr (\mathbf A D^2u).
\]
We also assume that the inhomogeneous term $f$ is zero.

We prove that if $D^2u(x^o)=0$ for some $x^o \in \Omega$, then $D^2u$ is differentiable $x^o$.
For simplicity, assume $x^o=0$ and $B_{1}(0) \subset \Omega$.

Let $\bar{\mathbf A}=(\mathbf A)_{B_r}$ denote the average of $\mathbf A$ over the ball $B_r=B_r(0)$, where $r \in (0,\frac12]$.
Decompose $u$ as $u=v+w$, where $w \in W^{2,p}(B_r) \cap W^{1,p}_0(B_r)$ (for some $p>1$) is the strong solution of the problem
\[
\tr(\bar{\mathbf A}D^2w) = -\tr((\mathbf{A}-\bar{\mathbf A})D^2u)\;\mbox{ in }\; B_r,\quad
w=0  \;\mbox{ on }\;\partial B_r.
\]

\begin{lemma}			\label{lem02}
Let $B=B_1(0)$ and $\mathbf A_0$ be a constant symmetric matrix satisfying the uniform ellipticity condition \eqref{ellipticity-nd}.
For $f \in L^p(B)$ with $p>1$, let $u \in W^{2,p}(B) \cap W^{1,p}_0(B)$ be the unique solution to the Dirichlet problem
\[
\left\{
\begin{aligned}
\tr(\mathbf A_0 D^2 u)&= f\;\mbox{ in }\; B,\\
u &= 0 \;\mbox{ on } \; \partial B.
\end{aligned}
\right.
\]
Then, for any $t>0$, the following estimate holds:
\[
\Abs{\{x \in B : \abs{D^2u(x)} > t\}}  \le \frac{C}{t} \int_{B} \abs{f},
\]
where $C=C(d, \lambda, \Lambda)$.
\end{lemma}

\begin{proof}
Refer to the proof of \cite[Lemma 3.3]{CDKK24} and \cite[Lemma 2.20]{DK17}.
\end{proof}

Applying Lemma \ref{lem02}, we establish that $D^2w \in L^{\frac12}(B_r)$, satisfying the estimate:
\begin{equation}			\label{eq2142fri}
\left(\fint_{B_r} \abs{D^2w}^{\frac12}\right)^{2} \le C \omega_{\mathbf A}(r) \norm{D^2u}_{L^\infty(B_r)},
\end{equation}
where $C=C(d,\lambda, \Lambda)$.

The function  $v=u-w$ satisfies
\[
\tr(\bar{\mathbf A} D^2 v)=0\;\text { in }\;B_r.
\]
By interior regularity for constant-coefficient elliptic equations, $v \in C^{\infty} (B_r)$, and
\[
\norm{D^2v}_{L^\infty(B_{r/2})} \le \frac{C}{r^2} \left(\fint_{B_r} \abs{v}^{\frac12}\right)^{2},
\]
where $C=C(d,\lambda, \Lambda)$.
This estimate remains valid if $v$ is replaced by $D^2v-\mathbf L$, where $\mathbf L$ is any affine symmetric matrix-valued function of the form
\[
\mathbf L(x)=x_1 \mathbf S^1 + \cdots +x_d \mathbf S^d + \mathbf P,\quad \text{with }\;\mathbf S^1, \ldots, \mathbf S^d, \mathbf P \in \mathbb{S}^d. 
\]

Let $\Sym^3(\bR^d)$ denote the space of symmetric $(0,3)$-tensors -- that is, fully symmetric trilinear forms -- over $\bR^d$.
Given $\mathcal{S} \in \Sym^3(\bR^d)$, we define its contraction with a vector $x \in \bR^d$ as the matrix
\[
\langle \mathcal{S}, x \rangle:= x_1 \mathbf S^1 + \cdots +x_d \mathbf S^d,
\]
where each slice $\mathbf S^k \in \mathbb{S}^d$ is defined by
\[
(\mathbf{S}^k)_{i,j}:=\mathcal{S}^{i,j,k}.
\]
In index notation, this contraction can be equivalently expressed as
\[
(\langle \mathcal{S}, x \rangle)_{i,j}=x_k\mathcal{S}^{i,j,k}. 
\]
With this notation, we define the affine space
\[
\mathfrak{S}=\left\{\mathbf L(x)= \langle \mathcal{S}, x \rangle + \mathbf{P}: \mathcal{S} \in \Sym^3(\bR^d),\; \mathbf{P} \in \mathbb{S}^d \right\}.
\]

Then, for any $\mathbf L \in \mathfrak{S}$, the following estimate holds:
\begin{equation}		\label{eq2135fri}
\norm{D^4v}_{L^\infty(B_{r/2})} \le \frac{C}{r^2} \left(\fint_{B_r} \abs{D^2v-\mathbf L}^{\frac12}\right)^{2}.
\end{equation}

By Taylor's theorem, for any $\rho \in (0, r]$,
\[
\sup_{x\in B_\rho}\;\abs{D^2v(x)-\langle D^3v(0), x \rangle - D^2v(0)}  \le C(d) \norm{D^4 v}_{L^\infty(B_\rho)} \,\rho^2,
\]
where $D^3v(0)$ is interpreted as an element of $\Sym^3(\bR^d)$.

Consequently, for any  $\kappa \in (0, \frac12)$ and any $\mathbf L \in \mathfrak{S}$, we obtain
\[
\left(\fint_{B_{\kappa r}} \abs{D^2v-\langle D^3v(0),x \rangle -D^2v(0)}^{\frac12}\right)^{2} \le C \norm{D^4v}_{L^\infty(B_{r/2})} (\kappa r)^2
\le C \kappa^2 \left(\fint_{B_r} \abs{D^2v-\mathbf L}^{\frac12}\right)^{2},
\]
where $C=C(d,\lambda, \Lambda)>0$.

We now define the function
\begin{equation}			\label{eq1100sat_nd}
\Phi(r):=\frac{1}{r} \,\inf_{\mathbf L \in \mathfrak{S}} \left(\fint_{B_r} \abs{D^2u-\mathbf L}^{\frac12}\right)^{2}.
\end{equation}
Since $u=v+w$, similar to \eqref{eq0224sat}, we obtain
\begin{equation*}
\kappa r \Phi(\kappa r) \le C \kappa^2 \left(\fint_{B_r} \abs{D^2u-\mathbf L}^{\frac12}\right)^{2} + C \left(\kappa^2+\kappa^{-2d}\right) \omega_{\mathbf A}(r) \norm{D^2u}_{L^\infty(B_r)}.
\end{equation*}

Let $\beta \in (0,1)$ be an arbitrary but fixed number.
Choose $\kappa=\kappa(d, \lambda,\Lambda, \beta) \in (0, \frac12)$ such that $C \kappa \le  \kappa^{\beta}$.
Then, we have
\begin{equation}			\label{eq1110sat_nd}
\Phi(\kappa r) \le \kappa^\beta \Phi(r) + C \omega_{\mathbf A}(r)\,\frac{1}{r} \norm{D^2u}_{L^\infty(B_r)},
\end{equation}
where $C=C(d, \lambda, \Lambda, \kappa)=C(d, \lambda, \Lambda, \beta)$.

Let $r_0 \in (0, \frac12]$ be a number to be chosen later.
By iterating \eqref{eq1110sat_nd}, we have, for $j=1,2,\ldots$,
\[
\Phi(\kappa^j r_0) \le \kappa^{\beta j} \Phi(r_0) + C \sum_{i=1}^{j} \kappa^{(i-1)\beta} \omega_{\mathbf A}(\kappa^{j-i} r_0)\,\frac{1}{\kappa^{j-i} r_0}\, \norm{D^2u}_{L^\infty(B_{\kappa^{j-i} r_0})}.
\]

By defining
\begin{equation}			\label{eq1244sat_nd}
\mathrm{M}_j(r_0):=\max_{0\le i < j}\, \frac{1}{\kappa^i r_0} \,\norm{D^2u}_{L^\infty(B_{\kappa^i r_0})} \quad \text{for }j=1,2, \ldots,
\end{equation}
we obtain
\begin{equation}			\label{eq1114sat_nd}
\Phi(\kappa^j r_0) \le \kappa^{\beta j} \Phi(r_0) + C \mathrm{M}_j(r_0) \tilde \omega_{\mathbf A}(\kappa^j r_0).
\end{equation}

For each fixed $r$, the infimum in \eqref{eq1100sat_nd} is realized by some $\mathbf L \in \mathfrak{S}$.
For each $j=0,1,2,\ldots$, let $\mathcal{S}_j \in \Sym^3(\bR^d)$ and $\mathbf P_j \in \mathbb{S}^d$ be chosen such that
\begin{equation}			\label{eq0203tue_nd}
\Phi(\kappa^j r_0)= \frac{1}{\kappa^j r_0}
\left(\fint_{B_{\kappa^j r_0}} \bigabs{D^2u-\langle \mathcal{S}_j, x\rangle -\mathbf P_j}^{\frac12}\right)^{2}.
\end{equation}

From \eqref{eq1100sat_nd} and H\"older's inequality, we have
\begin{equation}			\label{eq0538tue_nd}
\Phi(r_0) \le \frac{1}{r_0} \fint_{B_{r_0}} \abs{D^2u}.
\end{equation}
Combining \eqref{eq1114sat_nd} and \eqref{eq0538tue_nd}, we deduce
\begin{equation}\label{eq4.38_nd}
\Phi(\kappa^j r_0) \le  \frac{\kappa^{\beta j}}{r_0} \fint_{B_{r_0}} \abs{D^2u}+ C\mathrm{M}_j(r_0) \tilde \omega_{\mathbf A}(\kappa^j r_0).
\end{equation}

Next, observe that for $j=0,1,2,\ldots$, we have
\begin{equation}			\label{eq0218tue_nd}
\fint_{B_{\kappa^j r_0}} \bigabs{\langle \mathcal{S}_j,  x\rangle +\mathbf P_j}^{\frac12}
\le \fint_{B_{\kappa^j r_0}} \bigabs{D^2u-\langle \mathcal{S}_j,  x\rangle -\mathbf P_j}^{\frac12}+
\fint_{B_{\kappa^j r_0}} \abs{D^2u}^{\frac12}
\le 2\fint_{B_{\kappa^j r_0}} \abs{D^2u}^{\frac12}.
\end{equation}
Furthermore,
\begin{align*}
\abs{\mathbf P_j}^{\frac12} = \bigabs{\langle \mathcal{S}_j,  x\rangle + \mathbf P_j - 2(\langle \mathcal{S}_j,  x/2\rangle + \mathbf P_j)}^{\frac12} \le \bigabs{\langle \mathcal{S}_j,  x\rangle + \mathbf P_j}^{\frac12} + 2^{\frac12}\, \bigabs{\langle \mathcal{S}_j,  x/2\rangle + \mathbf P_j}^{\frac12}
\end{align*}
and
\[
\fint_{B_{\kappa^j r_0}} \bigabs{\langle \mathcal{S}_j,  x/2\rangle + \mathbf P_j}^{\frac12}  = \frac{2^d}{\abs{B_{\kappa^j r_0}}}\int_{B_{\kappa^j r_0/2}} \bigabs{\langle \mathcal{S}_j,  x\rangle + \mathbf P_j}^{\frac12}  \le 2^d\fint_{B_{\kappa^j r_0}} \bigabs{\langle \mathcal{S}_j,  x\rangle + \mathbf P_j}^{\frac12}.
\]
Thus,
\begin{equation}		\label{eq0213tue_nd}
\abs{\mathbf P_j} \le C \left(\fint_{B_{\kappa^j r_0}}\bigabs{\langle \mathcal{S}_j,  x\rangle + \mathbf P_j }^{\frac12}\right)^{2}\le C \left(\fint_{B_{\kappa^j r_0}} \abs{D^2u}^{\frac12}\right)^{2},\quad j=0,1,2,\ldots.
\end{equation}
Since $u \in C^2(\overline B_{1/2})$ by \cite[Theorem 1.6]{DK17} and $D^2u(0)=0$, the estimate \eqref{eq0213tue_nd} immediately implies
\begin{equation}			\label{eq0215tue_nd}
\lim_{j \to \infty} \mathbf P_j =0.
\end{equation}

\subsection*{Estimate of $\mathcal{S}_j$}
By the quasi-triangle inequality, we have
\[
\bigabs{\langle\mathcal{S}_j - \mathcal{S}_{j-1},x\rangle + \mathbf P_j -\mathbf P_{j-1}}^{\frac12} \le \bigabs{D^2u-\langle \mathcal{S}_j ,x\rangle- \mathbf P_j}^{\frac12} + \bigabs{D^2u-\langle \mathcal{S}_{j-1}, x\rangle- \mathbf P_{j-1}}^{\frac12}.
\]
Taking the average over $B_{\kappa^j r_0}$ and using the fact that $\abs{B_{\kappa^{j-1} r_0}}/ \abs{B_{\kappa^j r_0}} = \kappa^{-d}$, we obtain
\begin{equation}			\label{eq1949sat_nd}
\frac{1}{\kappa^j r_0}
\left(\fint_{B_{\kappa^j r_0}}\bigabs{\langle \mathcal{S}_j - \mathcal{S}_{j-1},x\rangle + \mathbf P_j -\mathbf P_{j-1}}^{\frac12}\right)^{2} \le C \Phi(\kappa^j r_0) + C \Phi(\kappa^{j-1} r_0)
\end{equation}
for $j=1,2,\ldots$, where $C=C(d, \lambda, \Lambda, \beta)$.

Next, observe that
\[
\abs{\mathbf P_j-\mathbf P_{j-1}}^{\frac12} = \bigabs{\langle \mathcal{S}_j - \mathcal{S}_{j-1},x\rangle + \mathbf P_j -\mathbf P_{j-1} - 2\left(\langle \mathcal{S}_j - \mathcal{S}_{j-1},x/2\rangle + \mathbf P_j -\mathbf P_{j-1}\right)}^{\frac12}.
\]
Using the quasi-triangle inequality, we obtain
\[
\abs{\mathbf P_j-\mathbf P_{j-1}}^{\frac12} \le \bigabs{\langle \mathcal{S}_j - \mathcal{S}_{j-1},x\rangle + \mathbf P_j -\mathbf P_{j-1}}^{\frac12} + 2^{\frac12} \,\bigabs{\langle\mathcal{S}_j - \mathcal{S}_{j-1},x/2\rangle + \mathbf P_j -\mathbf P_{j-1}}^{\frac12}.
\]
Moreover,
\[
\fint_{B_{\kappa^j r_0}} \bigabs{\langle \mathcal{S}_j - \mathcal{S}_{j-1},x/2\rangle + \mathbf P_j -\mathbf P_{j-1}}^{\frac12} \le
2^d\fint_{B_{\kappa^j r_0}} \bigabs{\langle \mathcal{S}_j -\mathcal{S}_{j-1},x \rangle + \mathbf P_j-\mathbf P_{j-1}}^{\frac12}.
\]
Combining these estimates, we conclude that
\[
\abs{\mathbf P_j-\mathbf P_{j-1}} \le C(d) \left(\fint_{B_{\kappa^j r_0}}\bigabs{\langle \mathcal{S}_j - \mathcal{S}_{j-1},x\rangle + \mathbf P_j -\mathbf P_{j-1}}^{\frac12}\right)^{2},\quad j=1,2,\ldots.
\]
Substituting this into \eqref{eq1949sat_nd}, we derive
\begin{equation}				\label{eq2247sat_nd}
\frac{1}{\kappa^j r_0} \abs{\mathbf P_j - \mathbf P_{j-1}} \le C \Phi(\kappa^j r_0)+C \Phi(\kappa^{j-1} r_0),\quad j=1,2,\ldots.
\end{equation}

On the other hand, for any $\mathcal{S} \in \Sym^3(\bR^d)$, we may write $\mathcal{S} = \abs{\mathcal{S}}\,\mathcal{T}$, where $\mathcal{T} \in \Sym^3(\bR^d)$ satisfies $\abs{\mathcal{T}} = 1$.
Then we have
\begin{equation}		\label{eq0229tue_nd}
\fint_{B_r} \abs{\langle \mathcal{S}, x\rangle}^{\frac12}  \ge \abs{\mathcal{S}}^{\frac12} \inf_{\abs{\mathcal T}=1}  \fint_{B_r}  \abs{\langle \mathcal{T}, x\rangle}^{\frac12} = \abs{\mathcal{S}}^{\frac12} \inf_{\abs{\mathcal T}=1} \fint_{B_1} r^{\frac12} \abs{\langle \mathcal{T}, x\rangle}^{\frac12} =C(d) r^{\frac12} \abs{\mathcal{S}}^{\frac12}.
\end{equation}

Then, by using \eqref{eq0229tue_nd}, the quasi-triangle inequality, \eqref{eq1949sat_nd}, \eqref{eq2247sat_nd},  \eqref{eq4.38_nd}, and the observation that $\mathrm{M}_{j-1}(r_0) \le \mathrm{M}_{j}(r_0)$, we obtain
\begin{align}
						\nonumber
\abs{\mathcal{S}_j-\mathcal{S}_{j-1}} &\le \frac{C}{\kappa^j r_0} \left(\fint_{B_{\kappa^j r_0}}\bigabs{\langle \mathcal{S}_j - \mathcal{S}_{j-1},x\rangle}^{\frac12}\right)^{2}\\
						\nonumber
& \le \frac{C}{\kappa^j r_0} \left(\fint_{B_{\kappa^j r_0}}\bigabs{\langle \mathcal{S}_j - \mathcal{S}_{j-1},x\rangle + \mathbf P_j -\mathbf P_{j-1}}^{\frac12}\right)^{2} + \frac{C}{\kappa^j r_0} \bigabs{\mathbf P_j - \mathbf P_{j-1}}\\
						\nonumber
&\le C \Phi(\kappa^j r_0)+C \Phi(\kappa^{j-1} r_0)\\
								\label{eq2316sat_nd}
&\le \frac{C\kappa^{\beta j}}{r_0}  \fint_{B_{r_0}} \abs{D^2u}+ C\mathrm{M}_j(r_0)\left\{\tilde \omega_{\mathbf A}(\kappa^j r_0) + \tilde \omega_{\mathbf A}(\kappa^{j-1} r_0)\right\}.
\end{align}

To estimate $\abs{\mathcal{S}_0}$, we proceed similarly to \eqref{eq2316sat_nd} by applying \eqref{eq0218tue_nd} and \eqref{eq0213tue_nd} with $j=0$, and using H\"older's inequality to obtain
\begin{equation}			\label{eq0230tue_nd}
\abs{\mathcal{S}_0} \le \frac{C}{r_0} \fint_{B_{r_0}} \abs{D^2u},
\end{equation}
where $C=C(d, \lambda, \Lambda, \beta)$.

For $k>l\ge 0$, we derive from \eqref{eq2316sat_nd} and the definition of $\mathrm{M}_j(r_0)$ that
\begin{align}
						\nonumber
\abs{\mathcal{S}_k -\mathcal{S}_l} &\le \sum_{j=l}^{k-1}\, \abs{\mathcal{S}_{j+1}-\mathcal{S}_j} \le  \sum_{j=l}^{k-1}  \frac{C\kappa^{\beta(j+1)}}{r_0} \fint_{B_{r_0}} \abs{D^2u}+ C\mathrm{M}_k(r_0) \sum_{j=l}^{k} \tilde \omega_{\mathbf A}(\kappa^j r_0) \\
						\label{eq0946tue_nd}
&\le  \frac{C \kappa^{\beta(l+1)}}{(1-\kappa^\beta)r_0} \fint_{B_{r_0}} \abs{D^2u}+ C\mathrm{M}_k(r_0) \int_0^{\kappa^l r_0} \frac{\tilde \omega_{\mathbf A}(t)}{t}\,dt.
\end{align}
In particular, by taking $k=j$ and $l=0$ in \eqref{eq0946tue_nd}, and using \eqref{eq0230tue_nd}, we obtain for $j=1,2,\ldots$ that
\begin{equation}			\label{eq0900tue_nd}
\abs{\mathcal{S}_j} \le \abs{\mathcal{S}_j-\mathcal{S}_0} + \abs{\mathcal{S}_0} \le \frac{C}{r_0} \fint_{B_{r_0}} \abs{D^2u}+ C\mathrm{M}_j(r_0) \int_0^{r_0} \frac{\tilde \omega_{\mathbf A}(t)}{t}\,dt.
\end{equation}

Similarly, we obtain from \eqref{eq2247sat} that for $k>l\ge 0$, we have
\begin{equation}			\label{eq0947tue_nd}
\abs{\mathbf P_k -\mathbf P_l} \le C \frac{\kappa^{(\beta+1)(l+1)}}{1-\kappa^{\beta+1}} \fint_{B_{r_0}} \abs{D^2u}+ C \kappa^l r_0 \mathrm{M}_k(r_0) \int_0^{\kappa^l r_0} \frac{\tilde \omega_{\mathbf A}(t)}{t}\,dt.
\end{equation}

\subsection*{Estimate for $\mathbf P_j$}
We shall derive improved estimates for $\abs{\mathbf P_j}$ using the following lemma, where we set 
\[
v(x):=u(x)-\frac{1}{6} \langle \mathcal{S}_j, x \rangle x\cdot x-\frac{1}{2} \mathbf P_j x \cdot x.
\]
Note that
\begin{equation}		\label{eq1852mon_nd}
D^2v(x)=D^2u(x)-\langle \mathcal{S}_j, x\rangle-\mathbf P_j.
\end{equation}
\begin{lemma}			\label{lem1702sat_nd}
For $0<r \le \frac12$, we have
\[
\sup_{B_{r}}\, \abs{D^2 v} \le C \left\{\left( \fint_{B_{2r}} \abs{D^2v}^{\frac12} \right)^{2} + (r \abs{\mathcal{S}_j}+\abs{\mathbf P_j}) \int_0^r \frac{\tilde \omega_{\mathbf A}(t)}{t}\,dt\right\},
\]
where $C=C(d, \lambda, \Lambda, \omega_{\mathbf A},\beta)$.
\end{lemma}

\begin{proof}
Since $\tr(\mathbf{A} D^2u)=0$ in $B_1$, it follows that
\[
\tr(\mathbf A D^2v) = -\tr (\mathbf A(\langle \mathcal{S}_j, x\rangle+\mathbf P_j))  \quad\text{in }\, B_{2r},
\]
for any $0<r \le \frac12$.
Let $x_0 \in B_{3r/2}$ and $0<t \le r/4$, and denote $\bar{\mathbf A}:=(\mathbf A)_{B_t(x_0)}$.

We decompose $v$ as $v=v_1+v_2$, where $v_1 \in W^{2,p} \cap W^{1,p}_0(B_t(x_0))$ (for some $p>1$) is the strong solution to the problem
\[
\tr(\bar{\mathbf A}D^2 v_1) = -\tr ( (\mathbf{A}-\bar{\mathbf A}) (\langle \mathcal{S}_j, x\rangle+\mathbf P_j)+ (\mathbf{A}- \bar{\mathbf A})D^2v)\quad\text{in }\,B_t(x_0),\qquad
\]
with boundary condition $v_1=0$ on $\partial B_t(x_0)$.

By Lemma \ref{lem02} and rescaling, we obtain
\[
\left(\fint_{B_t(x_0)} \abs{D^2v_1}^{\frac12} \right)^{2}
\le C \left(\fint_{B_t(x_0)}\abs{\mathbf A-\bar{\mathbf A}}\right) \left(r\abs{\mathcal{S}_j}+\abs{\mathbf P_j}\right)
+C \left(\fint_{B_t(x_0)}\abs{\mathbf A-\bar{\mathbf A}}\right)\norm{D^2v}_{L^\infty(B_t(x_0))},
\]
and thus, we have
\[
\left(\fint_{B_t(x_0)} \abs{D^2v_1}^{\frac12} \right)^{2} \le C \omega_{\mathbf A}(t)\left(r\abs{\mathcal{S}_j}+\abs{\mathbf P_j}\right)+ C  \omega_{\mathbf A}(t)\norm{D^2v}_{L^\infty(B_t(x_0))}.
\]

On the other hand, observe that $v_2=v-v_1$ satisfies
\[
\mathrm{L}_0v_2:= \tr(\bar{\mathbf A}D^2v_2) = -\tr \bar{\mathbf A} (\langle \mathcal{S}_j, x\rangle+\mathbf P_j) = \text{affine function} \quad\text{in }\,  B_t(x_0).
\]
Since $\mathrm{L}_0$ is a constant-coefficients operator, it follows that $\mathrm{L}_0 (D^2v_2)=0$.
The remainder of the proof then proceeds identically to that of \cite[Theorem 1.6]{DK17}.
\end{proof}

By Lemma \ref{lem1702sat_nd}, \eqref{eq1852mon_nd}, \eqref{eq0203tue_nd}, and \eqref{eq4.38_nd}, we have (cf. \eqref{eq0934wed0})
\begin{multline}				\label{eq0934wed0_nd}
 \norm{D^2u-\langle \mathcal{S}_j, x\rangle-\mathbf P_j}_{L^\infty(B_{\frac12 \kappa^j r_0})} \le C\kappa^{(1+\beta)j} \fint_{B_{r_0}} \abs{D^2u} + C\kappa^j r_0 \mathrm{M}_j(r_0) \tilde\omega_{\mathbf A}(\kappa^j r_0)\\
+ C \left(\kappa^j r_0\abs{\mathcal{S}_j}+\abs{\mathbf P_j}\right) \int_0^{\kappa^j r_0} \frac{\tilde\omega_{\mathbf A}(t)}{t}\,dt,\quad j=1,2,\ldots,
\end{multline}
where $C=C(d, \lambda, \Lambda, \omega_{\mathbf A}, \beta)$.
Using $D^2u(0)=0$, we infer from \eqref{eq0934wed0_nd} that
\begin{multline*}
\abs{\mathbf P_j} \le C\kappa^{(1+\beta)j} \fint_{B_{r_0}} \abs{D^2u} + C\kappa^j r_0 \mathrm{M}_j(r_0) \tilde\omega_{\mathbf A}(\kappa^j r_0)\\
+C \abs{\mathcal{S}_j}\kappa^j r_0 \int_0^{\kappa^j r_0} \frac{\tilde \omega_{\mathbf A}(t)}{t}\,dt
+C \abs{\mathbf P_j} \int_0^{\kappa^j r_0} \frac{\tilde\omega_{\mathbf A}(t)}{t}\,dt.
\end{multline*}
Let us fix $r_1(d, \lambda, \Lambda,\omega_{\mathbf A}, \beta) >0$ such that
\[
C \int_0^{r_1} \frac{\tilde \omega_{\mathbf A}(t)}{t}\,dt\le \frac12.
\]
We will later require $r_0 \le r_1$.
This ensures that
\[
\abs{\mathbf P_j} \le  C\kappa^{(1+\beta)j} \fint_{B_{r_0}} \abs{D^2u} + C\kappa^j r_0 \mathrm{M}_j(r_0) \tilde\omega_{\mathbf A}(\kappa^j r_0)
+C\abs{\mathcal{S}_j}\kappa^j r_0\int_0^{\kappa^j r_0} \frac{\tilde \omega_{\mathbf A}(t)}{t}\,dt.
\]
This, together with \eqref{eq0900tue_nd} and \eqref{rmk1147_2}, yields
\begin{equation}			\label{eq7.51_nd}
\abs{\mathbf P_j} \le C \kappa^{j} r_0 \left\{\kappa^{\beta j}+\int_0^{\kappa^j r_0} \frac{\tilde \omega_{\mathbf A}(t)}{t}\,dt\right\}\frac{1}{r_0} \fint_{B_{r_0}} \abs{D^2u} +C\kappa^j r_0 \mathrm{M}_j(r_0) \int_0^{\kappa^j r_0} \frac{\tilde \omega_{\mathbf A}(t)}{t}\,dt.
\end{equation}

\subsection*{Convergence of $\mathcal{S}_j$}
By \eqref{eq0934wed0_nd}, \eqref{eq0900tue_nd}, \eqref{eq7.51_nd}, and \eqref{rmk1147_2}, we have
\begin{multline}				\label{eq1920thu_nd}
\norm{D^2u-\langle \mathcal{S}_j,x\rangle-\mathbf P_j}_{L^\infty(B_{\frac12 \kappa^j r_0})}
\le C \kappa^j r_0 \left\{\kappa^{\beta j}+ \int_0^{\kappa^j r_0} \frac{\tilde \omega_{\mathbf A}(t)}{t}\,dt \right\}\frac{1}{r_0} \fint_{B_{r_0}} \abs{D^2u}\\
+C\kappa^j r_0 \mathrm{M}_j(r_0)  \int_0^{\kappa^j r_0} \frac{\tilde \omega_{\mathbf A}(t)}{t}\,dt.
\end{multline}
Then, from \eqref{eq1920thu_nd}, \eqref{eq0900tue_nd}, and \eqref{eq7.51_nd}, we infer that
\begin{equation}				\label{eq1555sun_nd}
\frac {1}{\kappa^{j} r_0}\norm{D^2u}_{L^\infty(B_{\frac12 \kappa^j r_0})}
\le \frac C {r_0} \fint_{B_{r_0}} \abs{D^2u} + C \mathrm{M}_j(r_0) \int_0^{r_0} \frac{\tilde \omega_{\mathbf A}(t)}{t}\,dt,
\end{equation}
where $C=C(d, \lambda, \Lambda, \omega_{\mathbf A}, \beta)$.

\begin{lemma}				\label{lem1548sun_nd}
There exists a constant $r_0=r_0(d, \lambda, \Lambda, \omega_{\mathbf A}, \beta) \in (0,\frac12)$ such that
\begin{equation}			\label{eq1951thu_nd}
\sup_{j \ge 1} \mathrm{M}_j(r_0) =\sup_{i \ge 0} \frac{1}{\kappa^i r_0} \,\norm{D^2u}_{L^\infty(B_{\kappa^i r_0})} \le \frac{C}{r_0} \fint_{B_{2r_0}} \abs{D^2u},
\end{equation}
where $C=C(d, \lambda, \Lambda, \omega_{\mathbf A}, \beta)$.
\end{lemma}

\begin{proof}
Refer to the proof of Lemma \ref{lem1548sun}.
\end{proof}

Now, Lemma \ref{lem1548sun_nd} and \eqref{eq0946tue_nd} imply that the sequence $\{\mathcal{S}_j\}$ is a Cauchy sequence in $\Sym^3(\bR^d)$, and thus $\mathcal{S}_j \to \mathcal{S}$ for some $\mathcal{S} \in \Sym^3(\bR^d)$.
Moreover, by taking the limit as $k\to \infty$ in \eqref{eq0946tue_nd} and \eqref{eq0947tue_nd} (while recalling \eqref{eq1951thu_nd} and \eqref{eq0215tue_nd}), respectively, and then setting $l=j$, we obtain the following estimates:
\begin{equation}	\label{eq1806sun_nd}
\begin{aligned}
\abs{\mathcal{S}_j-\mathcal{S}} &\le C \left\{ \kappa^{\beta j}+ \int_0^{\kappa^j r_0} \frac{\tilde \omega_{\mathbf A}(t)}{t}\,dt\right\}\frac{1}{r_0} \fint_{B_{2r_0}} \abs{D^2u},\\
\abs{\mathbf P_j}  &\le C \kappa^j \left\{\kappa^{\beta j} +\int_0^{\kappa^j r_0} \frac{\tilde \omega_{\mathbf A}(t)}{t}\,dt \right\}\fint_{B_{2r_0}} \abs{D^2u}.
\end{aligned}
\end{equation}

By the triangle inequality, \eqref{eq1920thu_nd}, \eqref{eq1806sun_nd}, and \eqref{eq1951thu_nd}, we obtain
\begin{align}			\nonumber
\norm{D^2u-\langle \mathcal{S}, x \rangle}_{L^\infty(B_{\frac12 \kappa^j r_0})}& \le \norm{D^2u- \langle \mathcal{S}_j, x\rangle -\mathbf  P_j}_{L^\infty(B_{\frac12 \kappa^j r_0})} + \frac{\kappa^j r_0}{2} \abs{\mathcal{S}_j - \mathcal{S}} +  \abs{\mathbf P_j}\\
					\label{eq2221sun_nd}
&\le C \kappa^j r_0\left\{\kappa^{\beta j}+\int_0^{\kappa^j r_0} \frac{\tilde\omega_{\mathbf A}(t)}{t}\,dt\right\} \frac{1}{r_0} \fint_{B_{2r_0}} \abs{D^2u}.
\end{align}

\subsection*{Conclusion}
It follows from \eqref{eq2221sun_nd} that
\begin{equation}		\label{eq1036wed_nd}
\frac{1}{r} \norm{D^2u-\langle \mathcal{S},x\rangle}_{L^\infty(B_r)}  \le \varrho_{\mathbf A}(r) \left( \frac{1}{r_0} \fint_{B_{2r_0}} \abs{D^2u} \right),
\end{equation}
where
\begin{equation}			\label{eq2143sat_nd}
\varrho_{\mathbf A}(r)=C\left\{\left(\frac{2r}{\kappa r_0}\right)^\beta+\int_0^{2r/\kappa} \frac{\tilde\omega_{\mathbf A}(t)}{t}\,dt\right\}.
\end{equation}
Note that $\varrho_{\mathbf A}$ is a modulus of continuity  determined by $d$, $\lambda$, $\Lambda$, $\omega_{\mathbf A}$, and $\beta \in (0,1)$.

In particular, we conclude from \eqref{eq1036wed_nd} that $D^2u$ is differentiable at $0$.
Moreover, it follows from \eqref{eq0900tue_nd} and \eqref{eq1951thu_nd} that (noting that $D^3u(0)=\mathcal{S}=\lim_{j\to \infty} \mathcal{S}_j$) we have
\[
\abs{D^3 u (0)} \le \frac{C}{r_0} \fint_{B_{2r_0}} \abs{D^2u},
\]
where $C=C(d, \lambda, \Lambda, \omega_{\mathbf A},\beta)$.

If $\mathbf A \in C^\alpha$ for some $\alpha \in (0,1)$, we have $\varrho_{\mathbf A}(r) \lesssim r^\alpha$ by choosing $\beta \in (\alpha,1)$ in \eqref{eq2143sat_nd}.
This complete the proof in the special case.
 \qed

\section{Proof of Theorem \ref{thm02}: General case}		\label{sec5}

We now proceed with the proof of Theorem \ref{thm02} in the general setting.

\smallskip
Define $\omega_{\rm coef}(\cdot)$ by
\begin{multline}		\label{omega_coef_nd}
\omega_{\rm coef}(r) := \omega_{\mathbf A}(r) +  r^{1-d} \sup_{x \in \Omega} \int_{\Omega \cap B_r(x)} \abs{\vec b} + r^{1-d} \sup_{x \in \Omega} \int_{\Omega \cap B_r(x)} \abs{D\vec b}\\
+r^{1-d} \sup_{x \in \Omega} \int_{\Omega \cap B_r(x)} \abs{c}+ r^{1-d} \sup_{x \in \Omega} \int_{\Omega \cap B_r(x)}\abs{Dc}.
\end{multline}
It is clear that $\omega_{\rm coef}(\cdot)$ satisfies the Dini condition:
\[
\int_0^1 \frac{\omega_{\rm coef}(t)}{t} \,dt<\infty.
\]
Although one might expect different powers of $r$ to appear in the definition \eqref{omega_coef_nd}, we retain the current form to maintain consistency with the notation  $\omega_{\rm coef}(\cdot, x_0)$ introduced in the proof of Lemma~\ref{lem2255sat}.

\smallskip
Let $u$ be a solution to \eqref{eq_main-nd}, meaning that $u$ satisfies
\[
\tr (\mathbf A D^2 u)+\vec b\cdot Du + cu=f\quad\text{in }\;\Omega.
\]

By \cite[Theorem 1.5]{DEK18} (see also \cite[Proposition 2.24]{DEK18}), we know that $D^2  u$ is continuous in $\Omega$.
Suppose $D^2u(x^o) = 0$ for some $x^o \in \Omega$.
As before, we may assume without loss of generality that $x^o = 0$ and $B_2(0) \subset \Omega$.

For $0<r \le \frac12$, denote $\bar{\mathbf A}:=(\mathbf A)_{B_r}$, the average of $\mathbf A$ over $B_r$, and define
\begin{align}
			\nonumber
\widetilde{\vec b \cdot Du}&= \left(\vec b -(\vec b)_{B_r} -(D \vec b)_{B_r} x\right)\cdot Du+ (\vec b)_{B_r}\cdot \left(Du-(Du)_{B_r}\right)+\left(Du - (Du)_{B_r}\right)\cdot  (D \vec b)_{B_r}\,x,\\
			\nonumber
\widetilde{cu}&=\left(c- (c)_{B_r}- (Dc)_{B_r}\cdot x\right) u+ (c)_{B_r} \left(u-(u)_{B_r}-(Du)_{B_r}\cdot x\right)+\left(u - (u)_{B_r}\right) (Dc)_{B_r}\cdot x,\\
			\label{eq1604sat}
\widetilde{f}&=f- (f)_{B_r} - (Df)_{B_r} \cdot x.
\end{align}

We decompose $u$ as $u=v+w$, where $w \in W^{2,p} \cap W^{1,p}_0(B_r)$ (for some $p>1$) is the solution of the problem
\[
\tr( \bar{\mathbf A} D^2w) = -\tr((\mathbf{A}-\bar{\mathbf A})D^2u)+ \widetilde f- \widetilde{\vec b \cdot Du} -\widetilde{cu}\;\mbox{ in }\; B_r,\quad
w=0 \;\mbox{ on }\;\partial B_r.
\]

By Lemma \ref{lem02}, we obtain the following estimate via rescaling:
\begin{equation}			\label{eq1358wed}
\left(\fint_{B_r} \abs{D^2w}^{\frac12}\,dx\right)^{2} \le C\omega_{\mathbf A}(r) \norm{D^2u}_{L^\infty(B_r)}+C \fint_{B_r} \abs{\widetilde{f}}+\fint_{B_r} \abs{\widetilde{\vec b \cdot Du}} +\fint_{B_r} \abs{\widetilde{cu}}.
\end{equation}
Note that the Poincar\'e inequality implies
\begin{equation}			\label{eq1357wed}
\fint_{B_r} \abs{f- (f)_{B_r} - (Df)_{B_r} \cdot x} \le C r \fint_{B_r} \abs{Df- (Df)_{B_r}}.
\end{equation}
Therefore, we obtain
\[
\fint_{B_r} \abs{\widetilde f} \le Cr \omega_{Df}(r).
\]
By applying \eqref{eq1357wed} to $\vec b$, $c$, and $u$, we also derive the following estimates:
\begin{align*}
\fint_{B_r} \abs{\widetilde{\vec b \cdot Du}}& \le C r\omega_{D \vec b}(r)\norm{Du}_{L^\infty(B_r)}+Cr \norm{D^2u}_{L^\infty(B_r)} \fint_{B_r} \abs{\vec b}+ Cr^2 \norm{D^2u}_{L^\infty(B_r)} \fint_{B_r} \abs{D\vec b},\\
\fint_{B_r} \abs{\widetilde{cu}} &\le Cr\omega_{Dc}(r)\norm{u}_{L^\infty(B_r)}+ C r^2\norm{D^2u}_{L^\infty(B_r)}\fint_{B_r} \abs{c}+ C r^2 \norm{Du}_{L^\infty(B_r)}\fint_{B_r} \abs{Dc}.
\end{align*}

We define $\omega_{\rm lot}(r)$ as follows:
\begin{equation}		\label{omega_lot}
\omega_{\rm lot}(r) :=\norm{Du}_{L^\infty(B_r)} \left(\omega_{D \vec b}(r) + r^{1-d} \sup_{x\in \Omega} \int_{\Omega \cap B_r(x)} \abs{Dc} \right)+ \norm{u}_{L^\infty(B_r)}\omega_{Dc}(r).
\end{equation}
It is clear that $\omega_{\rm lot}(\cdot)$ satisfy the Dini condition:
\[
\int_0^1 \frac{\omega_{\rm lot}(t)}{t} \,dt  <\infty.
\]
Considering the  definitions \eqref{omega_coef_nd}, \eqref{omega_lot}, and above estimates, we obtain from \eqref{eq1358wed} that:
\begin{equation}			\label{eq2358wed}
\left(\fint_{B_r} \abs{D^2w}^{\frac12}\right)^{2} \le C \omega_{\rm coef}(r) \norm{D^2u}_{L^\infty(B_r)} + Cr \omega_{Df}(r)+Cr\omega_{\rm lot}(r).
\end{equation}

It is clear that $f-\widetilde f$ is an affine function. 
Similarly, observe that both $cu- \widetilde{cu}$ and $\vec b\cdot Du- \widetilde{\vec b\cdot Du}$ are affine functions as well.
Therefore, $v=u-w$ satisfies
\[
\mathrm{L}_0 v:=\tr(\bar{\mathbf A} D^2v)= \text{affine function}\;\text { in }\;B_r.
\]
Note that $\mathrm{L}_0(D^2v-\mathbf{L})=0$ in $B_r$ for any $\mathbf L \in \mathfrak{S}$.
Therefore, the same reasoning that led to estimate \eqref{eq2135fri} also applies here, yielding the identical estimate.

Let $\Phi(r)$ be defined as in \eqref{eq1100sat_nd}, and let $\beta \in (0,1)$ be an arbitrary but fixed constant.
By employing \eqref{eq2358wed} instead of \eqref{eq2142fri}, we derive an estimate similar to \eqref{eq1110sat_nd} (cf. \eqref{eq2154thu}). Specifically, there exists a constant $\kappa=\kappa(d,\lambda, \Lambda, \beta) \in (0,\frac12)$ such that
\begin{equation*}
\Phi(\kappa r) \le \kappa^\beta \Phi(r) + C \omega_{\rm coef}(r)\, \frac{1}{r} \norm{D^2u}_{L^\infty(B_r)} +C \omega_{Df}(r) + C \omega_{\rm lot}(r),
\end{equation*}
where $C=C(d, \lambda, \Lambda, \beta)$.

Let $\mathrm{M}_j(r_0)$ be as defined in \eqref{eq1244sat_nd}, where $r_0\in (0, \frac{1}{2}]$ is a number to be chosen later.
Then we obtain, similar to \eqref{eq4.38_nd}, the following estimate:
\begin{equation}			\label{eq1052sat}
\Phi(\kappa^j r_0) \le  \frac{\kappa^{\beta j}}{r_0} \fint_{B_{r_0}} \abs{D^2u}+ C\mathrm{M}_j(r_0) \tilde \omega_{\rm coef}(\kappa^j r_0) +C \tilde \omega_{Df}(\kappa^j r_0)\\
+C \tilde \omega_{\rm lot}(\kappa^j r_0).
\end{equation}

Let $\mathcal{S}_j \in \Sym^3(\mathbb{R}^d)$ and $\mathbf{P}_j \in \mathbb{S}^d$ for $j=0,1,2,\ldots$ be chosen as in \eqref{eq0203tue_nd}.
By utilizing \eqref{eq1052sat} instead of \eqref{eq4.38_nd}, we conclude (cf. \eqref{eq0215tue_nd}) that:
\begin{equation}			\label{eq1717sun_nd}
\lim_{j\to \infty} \mathbf{P}_j=0.
\end{equation}

We similarly obtain the following estimate for $k>l\ge 0$ (cf. \eqref{eq0946tue_nd} and \eqref{eq0947tue_nd}):
\begin{multline}			\label{eq1718sun_nd}
\abs{\mathcal{S}_k -\mathcal{S}_l} +\frac{\abs{\mathbf{P}_k - \mathbf{P}_l}}{\kappa^l r_0}
\le \frac{C \kappa^{\beta l}}{r_0} \fint_{B_{r_0}} \abs{D^2u}+ C\mathrm{M}_k(r_0) \int_0^{\kappa^l r_0} \frac{\tilde \omega_{\rm coef}(t)}{t}\,dt\\
+C \int_0^{\kappa^l r_0} \frac{\tilde \omega_{Df}(t)}{t}\,dt+C \int_0^{\kappa^l r_0} \frac{\tilde \omega_{\rm lot}(t)}{t}\,dt,
\end{multline}
where $C=C(d, \lambda, \Lambda,\beta)$.
This also yields the following bound for $\mathcal{S}_j$ (cf. \eqref{eq0900tue_nd}):
\begin{multline}			\label{eq1628sun_nd}
\abs{\mathcal{S}_j} \le  \frac{C}{r_0} \fint_{B_{r_0}} \abs{D^2u}+ C\mathrm{M}_j(r_0) \int_0^{r_0} \frac{\tilde \omega_{\rm coef}(t)}{t}\,dt\\
+C \int_0^{r_0} \frac{\tilde \omega_{Df}(t)}{t}\,dt +C \int_0^{r_0} \frac{\tilde \omega_{\rm lot}(t)}{t}\,dt,
\end{multline}
where $C=C(d, \lambda, \Lambda, \beta)$.

\medskip
The following lemma serves as a counterpart to Lemma \ref{lem3.11sat}.
\begin{lemma}			\label{lem2255sat}
Let $v$ be defined by
\begin{equation}			\label{eq1144sat}
v(x):=u(x)-\frac{1}{6} \langle \mathcal{S}_j, x \rangle x\cdot x-\frac{1}{2} \mathbf P_j x \cdot x.
\end{equation}
Then, for any $0<r \le \frac12$, we have
\begin{multline*}
\sup_{B_{r}}\, \abs{D^2v} \le C\left( \fint_{B_{2r}} \abs{D^2v}^{\frac12} \right)^{2} + C r\int_0^r \frac{\tilde \omega_{Df}(t)}{t}\,dt +Cr \int_0^r \frac{\tilde \omega_{\rm lot}(t)}{t}\,dt \\
+C(r\abs{\mathcal{S}_j}+\abs{\mathbf{P}_j})\int_0^r \frac{\tilde \omega_{\rm coef}(t)}{t}\,dt,
\end{multline*}
where $C=C(d, \lambda, \Lambda, \omega_{\rm coef},\beta)$.
\end{lemma}
\begin{proof}
For $x_0 \in B_{3r/2}$ and $0<t \le r/4$, define $\bar{\mathbf A}:=(\mathbf A)_{B_t(x_0)}$ as the average of $\mathbf A$ over the ball  $B_t(x_0)$.
Note that $v$ satisfies
\begin{multline*}
\tr(\bar{\mathbf A}D^2v)
= f-\vec b \cdot Dv -cv-\tr((\mathbf A-\bar{\mathbf A})D^2 v) -\tr(\mathbf A (\langle \mathcal{S}_j, x\rangle + \mathbf{P}_j))\\
-\vec b\cdot(\tfrac{1}{2}\langle\mathcal{S}_j, x\rangle x + \mathbf{P}_jx) -c\left(\tfrac{1}{6} \langle \mathcal{S}_j, x \rangle x\cdot x+\tfrac{1}{2} \mathbf P_j x \cdot x\right).
\end{multline*}

Define the functions $\widetilde{f}$, $\widetilde{\vec b \cdot Dv}$, and $\widetilde{cv}$ as in \eqref{eq1604sat}.
More precisely, we set:
\begin{align}
			\nonumber
\widetilde{\vec b \cdot Dv}&= \left(\vec b -(\vec b)_{B_t(x_0)} -(D \vec b)_{B_t(x_0)}(x-x_0)\right)\cdot Dv+ (\vec b)_{B_t(x_0)}\cdot \left(Dv-(Dv)_{B_t(x_0)}\right)\\
			\nonumber
&\qquad+\left(Dv - (Dv)_{B_t(x_0)}\right)\cdot  (D \vec b)_{B_t(x_0)}\,(x-x_0),\\
			\nonumber
\widetilde{cv}&=\left(c- (c)_{B_t(x_0)}- (Dc)_{B_t(x_0)}\cdot(x-x_0)\right) v+\left(v - (v)_{B_t(x_0)}\right) (Dc)_{B_t(x_0)}\cdot (x-x_0)\\
			\nonumber
&\qquad + (c)_{B_t(x_0)} \left(v-(v)_{B_t(x_0)}-(Dv)_{B_t(x_0)}\cdot(x-x_0)\right),\\
			\label{eq1555sat}
\widetilde{f}&=f- (f)_{B_t(x_0)} - (Df)_{B_t(x_0)} \cdot (x-x_0).
\end{align}
Note that the differences $f-\widetilde f$, $cv- \widetilde{cv}$ and $\vec b\cdot Dv- \widetilde{\vec b\cdot Dv}$ are all affine functions.

\smallskip
We decompose $v$ as $v=v_1+v_2$, where $v_1 \in W^{2,p} \cap W^{1,p}_0(B_t(x_0))$ for some $p>1$, and it solves the following Dirichlet problem:
\begin{align*}
\tr(\bar{\mathbf A}D^2v_1) &= \widetilde{f}- \widetilde{\vec b \cdot Dv}-\widetilde{cv}-\tr((\mathbf A-\bar{\mathbf A})D^2 v)-\tr\{(\mathbf A-\bar{\mathbf A}) (\langle \mathcal{S}_j, x\rangle + \mathbf{P}_j)\} \\
&\quad-\vec b \cdot \left\{ (\tfrac{1}{2}\langle\mathcal{S}_j, x\rangle x + \mathbf{P}_jx)- (\tfrac{1}{2}\langle\mathcal{S}_j, x_0 \rangle x_0 + \mathbf{P}_j x_0) \right\}\\
&\quad-(\vec b-(\vec b)_{B_t(x_0)}) \cdot(\tfrac{1}{2}\langle\mathcal{S}_j, x_0\rangle x_0 + \mathbf{P}_j x_0) \\
&\quad -c\left\{ (\tfrac{1}{6} \langle \mathcal{S}_j, x \rangle x\cdot x+\tfrac{1}{2} \mathbf P_j x \cdot x)-(\tfrac{1}{6} \langle \mathcal{S}_j, x_0 \rangle x_0\cdot x_0+\tfrac{1}{2} \mathbf P_j x_0 \cdot x_0)\right\}\\
&\quad -(c-(c)_{B_t(x_0)}) (\tfrac{1}{6} \langle \mathcal{S}_j, x_0 \rangle x_0\cdot x_0+\tfrac{1}{2} \mathbf P_j x_0 \cdot x_0) \quad\mbox{in }\,B_t(x_0),
\end{align*}
with $v_1=0$ on $\partial B_t(x_0)$.

\smallskip
Then, by applying Lemma \ref{lem02} with a rescaling argument, and using the mean value theorem together with the Poincar\'e inequality, we obtain the following estimate:
\begin{align}
				\nonumber
\left(\fint_{B_t(x_0)} \abs{D^2v_1}^{\frac12} \right)^{2} &\lesssim \fint_{B_t(x_0)} \left(\abs{\widetilde{f}}+ \abs{\widetilde{\vec b \cdot Dv}}+\abs{\widetilde{cv}} \right)
+\omega_{\mathbf A}(t) \norm{D^2v}_{L^\infty(B_t(x_0))}\\
				\nonumber
&\qquad+\omega_{\mathbf A}(t)(r\abs{\mathcal{S}_j}+\abs{\mathbf P_j})+ \left(\fint_{B_t(x_0)}\abs{\vec b}\right)t(r\abs{\mathcal{S}_j}+\abs{\mathbf P_j})\\
				\nonumber
&\qquad+ t\left(\fint_{B_t(x_0)}\abs{D\vec b}\right) r(r\abs{\mathcal{S}_j}+\abs{\mathbf P_j}) +t\left(\fint_{B_t(x_0)}\abs{c}\right) r(r\abs{\mathcal{S}_j}+\abs{\mathbf P_j})\\
				\label{eq1636sun}
&\qquad+t\left(\fint_{B_t(x_0)}\abs{Dc}\right)r^2(r\abs{\mathcal{S}_j}+\abs{\mathbf P_j}).
\end{align}

Next, combining \eqref{eq1555sat} and \eqref{eq1357wed} with \eqref{eq1144sat}, we obtain
\begin{align}
				\nonumber
\fint_{B_t(x_0)}\abs{\widetilde{f}} &+ \abs{\widetilde{\vec b \cdot Dv}}+\abs{\widetilde{cv}} \lesssim t \omega_{Df}(t)+t\omega_{D \vec b}(t)\norm{Du}_{L^\infty(B_t(x_0))}+tr(r\abs{\mathcal{S}_j}+\abs{\mathbf P_j}) \fint_{B_t(x_0)} \abs{D\vec b}\\
				\nonumber
&\quad+t \norm{D^2v}_{L^\infty(B_t(x_0))} \fint_{B_t(x_0)} \abs{\vec b}+ t^2 \norm{D^2v}_{L^\infty(B_t(x_0))} \fint_{B_t(x_0)} \abs{D\vec b}\\
				\nonumber
&\quad+t\omega_{Dc}(t)\norm{u}_{L^\infty(B_t(x_0))}+ tr^2(r\abs{\mathcal{S}_j}+\abs{\mathbf P_j})\fint_{B_t(x_0)}\abs{Dc}+ t^2\norm{Du}_{L^\infty(B_t(x_0))}  \fint_{B_t(x_0)}\abs{Dc}\\
				\label{eq1635sun}
&\quad +t^2 r(r\abs{\mathcal{S}_j}+\abs{\mathbf P_j}) \fint_{B_t(x_0)}\abs{Dc}+t^2\norm{D^2v}_{L^\infty(B_t(x_0))} \fint_{B_t(x_0)}\abs{c}.
\end{align}
Define (cf. \eqref{omega_lot})
\begin{equation*}		
\omega_{\rm lot}(t,x_0):=\norm{Du}_{L^\infty(B_t(x_0))} \left(\omega_{D \vec b}(t) +t^{1-d} \int_{B_t(x_0)} \abs{Dc} \right)+ \norm{u}_{L^\infty(B_t(x_0))} \omega_{Dc}(t).
\end{equation*}
Then, using \eqref{eq1636sun} and \eqref{eq1635sun}, we obtain
\begin{multline*}
\left(\fint_{B_t(x_0)} \abs{Dv_1}^{\frac12} \right)^{2} \le Ct \omega_{Df}(t) +C \omega_{\rm coef}(t)\norm{D^2v}_{L^\infty(B_t(x_0))} + Ct \omega_{\rm lot}(t, x_0)\\
+C(r\abs{\mathcal{S}_j}+\abs{\mathbf{P}_j})\omega_{\rm coef}(t,x_0),
\end{multline*}
where $C=C(d, \lambda, \Lambda)$, and we define  (noting that $r\le 1$)
\[
\omega_{\rm coef}(t,x_0):=\omega_{\mathbf A}(t)+t\fint_{B_t(x_0)}\abs{\vec b} + t \fint_{B_t(x_0)}\abs{c} + t\fint_{B_t(x_0)}\abs{D \vec b}+ t\fint_{B_t(x_0)}\abs{Dc}.
\]

On the other hand, recall that $f-\widetilde f$, $cv- \widetilde{cv}$ and $\vec b\cdot Dv- \widetilde{\vec b\cdot Dv}$ are all affine functions.
Thus, $v_2=v-v_1$ satisfies
\[
\mathrm{L}_0 v_2:=\tr(\bar{\mathbf A}D^2 v_2) = \mbox{affine function} \quad\mbox{ in }\, B_t(x_0).
\]
Therefore, $\mathrm{L}_0(D^2 v_2)=0$ in $B_t(x_0)$.
The remainder of the proof follows the same arguments as those in the proof of \cite[Theorem 1.6]{DK17} and \cite[Lemma 2.2]{KL21}.
\end{proof}

By using Lemma~\ref{lem2255sat}, \eqref{eq1052sat} and \eqref{eq1144sat}, we obtain (cf. \eqref{eq0934wed0_nd})
\begin{align}	
			\nonumber
&\norm{D^2u-\langle \mathcal{S}_j, x \rangle-\mathbf P_j}_{L^\infty(B_{\frac12 \kappa^j r_0})} \le C\kappa^{(1+\beta)j} \fint_{B_{r_0}} \abs{D^2u}+C\kappa^j r_0 \mathrm{M}_j(r_0) \tilde\omega_{\rm coef}(\kappa^j r_0)\\
			\nonumber
&\qquad\qquad+C \left(\abs{\mathcal{S}_j}\kappa^j r_0+\abs{\mathbf{P}_j}\right) \int_0^{\kappa^j r_0}\frac{\tilde \omega_{\rm coef}(t)}{t}\,dt + C \kappa^j r_0 \int_0^{\kappa^j r_0} \frac{\tilde \omega_{Df}(t)}{t}\,dt \\
			\label{eq1218mon_nd}
&\qquad\qquad +C\kappa^j r_0 \int_0^{\kappa^j r_0} \frac{\tilde \omega_{\rm lot}(t)}{t}\,dt.
\end{align}

Using $D^2u(0)=0$, we deduce from \eqref{eq1218mon_nd} that
\begin{multline}				\label{eq1123sun}
\abs{\mathbf P_j} \le C\kappa^{(1+\beta)j} \fint_{B_{r_0}} \abs{D^2u}+C\kappa^j r_0 \mathrm{M}_j(r_0) \tilde\omega_{\rm coef}(\kappa^j r_0) +C \kappa^j r_0 \abs{\mathcal{S}_j}\int_0^{\kappa^j r_0}\frac{\tilde \omega_{\rm coef}(t)}{t}\,dt \\
+ C \abs{\mathbf{P}_j} \int_0^{\kappa^j r_0}\frac{\tilde \omega_{\rm coef}(t)}{t}\,dt+C \kappa^j r_0 \int_0^{\kappa^j r_0} \frac{\tilde \omega_{Df}(t)}{t}\,dt +C\kappa^j r_0 \int_0^{\kappa^j r_0} \frac{\tilde \omega_{\rm lot}(t)}{t}\,dt.
\end{multline}

We require $r_0 \le r_1$, where $r_1= r_1(d, \lambda, \Lambda, \omega_{\rm coef}, \beta)>0$ is chosen so that
\[
C \int_0^{r_1} \frac{\tilde \omega_{\rm coef}(t)}{t}\,dt\le \frac12.
\]

We derive from \eqref{eq1123sun} and \eqref{eq1628sun_nd} the following inequality:
\begin{align}
			\nonumber
\abs{\mathbf{P}_j} & \le C {\kappa^j r_0}\left\{\kappa^{\beta j}+\int_0^{\kappa^j r_0} \frac{\tilde \omega_{\rm coef}(t)}{t}\,dt\right\}\frac{1}{r_0} \fint_{B_{r_0}} \abs{D^2u}+C{\kappa^j r_0}\mathrm{M}_j(r_0)\int_0^{\kappa^j r_0}\frac{\tilde\omega_{\rm coef}(t)}{t}\,dt\\
			\label{eq1123mon}
&\qquad+ C\kappa^j r_0 \int_0^{r_0}\frac{\tilde\omega_{Df}(t)}{t}\,dt + C\kappa^j r_0 \int_0^{r_0} \frac{\tilde \omega_{\rm lot}(t)}{t}\,dt,
\end{align}
where $C=C(d, \lambda, \Lambda,\omega_{\rm coef}, \beta)$.

Then, similar to \eqref{eq1920thu_nd}, we obtain from \eqref{eq1218mon_nd}, \eqref{eq1628sun_nd}, and \eqref{eq1123mon} the following estimate:
\begin{align}				\nonumber
&\norm{D^2u-\langle\mathcal{S}_j, x\rangle-\mathbf{P}_j}_{L^\infty(B_{\frac12 \kappa^j r_0})}  \le C \kappa^j r_0 \left\{\kappa^{\beta j}+ \int_0^{\kappa^j r_0} \frac{\tilde \omega_{\rm coef}(t)}{t}\,dt \right\} \frac{1}{r_0} \fint_{B_{r_0}} \abs{D^2u}\\
						\nonumber
&\qquad+C \kappa^j r_0 \mathrm{M}_j(r_0) \int_0^{\kappa^j r_0} \frac{\tilde\omega_{\rm coef}(t)}{t}\,dt +C \kappa^j r_0 \left\{\int_0^{\kappa^j r_0} \frac{\tilde \omega_{Df}(t)}{t}\,dt +\int_0^{\kappa^j r_0}\frac{\tilde \omega_{\rm lot}(t)}{t}\,dt\right\}.\\
						\label{eq1720sun}
&\qquad +C \kappa^j r_0 \left\{ \int_0^{r_0} \frac{\tilde \omega_{Df}(t)}{t}\,dt +\int_0^{r_0} \frac{\tilde \omega_{\rm dat}(t)}{t}\,dt\right\}\int_0^{\kappa^j r_0} \frac{\tilde \omega_{\rm coef}(t)}{t}\,dt.
\end{align}
Additionally, similar to \eqref{eq1555sun_nd}, we obtain
\begin{multline}				\label{eq1244mon}
\frac 2 {\kappa^j r_0}\norm{D^2u}_{L^\infty(B_{\frac12 \kappa^j r_0})}
\le \frac C {r_0} \fint_{B_{r_0}} \abs{D^2u} + C \mathrm{M}_j(r_0) \int_0^{r_0} \frac{\tilde \omega_{\rm coef}(t)}{t}\,dt \\
+ C \int_0^{r_0}\frac{\tilde\omega_{Df}(t)}{t}\,dt + C \int_0^{r_0}\frac{\tilde\omega_{\rm lot}(t)}{t}\,dt,
\end{multline}
where $C=C(d, \lambda, \Lambda, \omega_{\rm coef}, \beta)$.

Then, by following the same proof as in Lemma~\ref{lem1548sun_nd}, we conclude from \eqref{eq1244mon} that there exists $r_0=r_0(d, \lambda, \Lambda, \omega_{\rm coef}, \beta) \in (0,\frac12]$ such that
\begin{equation}			\label{eq1706sun}
\sup_{j \ge 1} \mathrm{M}_j(r_0) \le \frac{C}{r_0} \fint_{B_{2r_0}} \abs{D^2u}+C \int_0^{r_0} \frac{\tilde \omega_{Df}(t)}{t}\,dt + C \int_0^{r_0} \frac{\tilde \omega_{\rm lot}(t)}{t}\,dt,
\end{equation}
where $C=C(d, \lambda, \Lambda, \omega_{\rm coef}, \beta)$.

From \eqref{eq1718sun_nd} and \eqref{eq1706sun}, we can conclude that the sequence $\{\mathcal{S}_j\}$ is a Cauchy sequence in $\Sym^3(\mathbb{R}^d)$. Therefore, it converges to some $\mathcal{S} \in \Sym^3(\mathbb{R}^d)$.
By taking the limit as $k\to \infty$ in \eqref{eq1718sun_nd} (while recalling \eqref{eq1706sun} and \eqref{eq1717sun_nd}), and then setting $l=j$, we obtain the following estimate:
\begin{align}
			\nonumber
\abs{\mathcal{S}_j -\mathcal{S}} + \frac{\abs{\mathbf{P}_j}}{\kappa^j r_0} &\le C \left\{ \kappa^{\beta j}+\int_0^{\kappa^j r_0} \frac{\tilde \omega_{\rm coef}(t)}{t}\,dt \right\} \frac{1}{r_0} \fint_{B_{2r_0}} \abs{D^2u}\\
			\nonumber
&\qquad+C \left\{\int_0^{r_0} \frac{\tilde \omega_{Df}(t)}{t}\,dt +\int_0^{r_0} \frac{\tilde \omega_{\rm lot}(t)}{t}\,dt \right\} \int_0^{\kappa^j r_0} \frac{\tilde \omega_{\rm coef}(t)}{t}\,dt\\
			\label{eq2203sun}
&\qquad+ C \int_0^{\kappa^j r_0} \frac{\tilde \omega_{Df}(t)}{t}\,dt +C \int_0^{\kappa^j r_0} \frac{\tilde \omega_{\rm lot}(t)}{t}\,dt,
\end{align}
where $C=C(d, \lambda, \Lambda,\omega_{\rm coef}, \beta)$.

Then, similar to \eqref{eq2221sun_nd}, it follows from \eqref{eq1720sun}, \eqref{eq1706sun}, and \eqref{eq2203sun} that
\begin{align*}
\norm{D^2u-\langle \mathcal{S}, x\rangle}_{L^\infty(B_{\frac12 \kappa^j r_0})} &\le
C \kappa^j r_0 \left\{\kappa^{\beta j}+ \int_0^{\kappa^j r_0} \frac{\tilde \omega_{\rm coef}(t)}{t}\,dt \right\}\frac{1}{r_0} \fint_{B_{2r_0}} \abs{D^2u}\\
&\qquad+C \kappa^j r_0 \left\{\int_0^{r_0} \frac{\tilde \omega_{Df}(t)}{t}\,dt +\int_0^{r_0} \frac{\tilde \omega_{\rm lot}(t)}{t}\,dt \right\} \int_0^{\kappa^j r_0} \frac{\tilde \omega_{\rm coef}(t)}{t}\,dt\\
&\qquad+ C \kappa^j r_0 \left\{\int_0^{\kappa^j r_0} \frac{\tilde \omega_{Df}(t)}{t}\,dt  \int_0^{\kappa^j r_0} \frac{\tilde \omega_{\rm lot}(t)}{t}\,dt \right\}.
\end{align*}

From the previous inequality, we derive the following uniform estimate for all $r\in (0, r_0/2)$:
\begin{multline}		\label{eq2010sun}
\frac{1}{r} \norm{D^2u-\langle \mathcal{S}, x\rangle}_{L^\infty(B_r)}  \le \varrho_{\rm coef}(r) \left\{\frac{1}{r_0} \fint_{B_{2r_0}} \abs{D^2u} + \varrho_{Df}(r_0)+\varrho_{\rm lot}(r_0)\right\}\\
+\varrho_{Df}(r)+\varrho_{\rm lot}(r),
\end{multline}
where the moduli of continuity $\varrho_{\cdots}(r)$ are defined by
\begin{equation*}		
\begin{aligned}
\varrho_{\rm coef}(r)&:=C\left\{\left(\frac{2r}{\kappa r_0}\right)^\beta+\int_0^{2r/\kappa}\frac{\tilde\omega_{\rm coef}(t)}{t}\,dt\right\},\\
\varrho_{Df}(r)&:=C \int_0^{2r/\kappa} \frac{\tilde \omega_{Df}(t)}{t}\,dt,\\
\varrho_{\rm lot}(r)&:=C \int_0^{2r/\kappa} \frac{\tilde \omega_{\rm lot}(t)}{t}\,dt,
\end{aligned}
\end{equation*}
where $C=C(d, \lambda, \Lambda,\omega_{\rm coef}, \beta)$.

In particular, from \eqref{eq2010sun}, we conclude that $D^2u$ is differentiable at $0$.
Moreover, it follows from \eqref{eq1628sun_nd} and \eqref{eq1706sun} that (noting $D^3u(0)=\mathcal{S}=\lim_{j\to \infty} \mathcal{S}_j$):
\[
\abs{D^3 u (0)} \le C \left\{\frac{1}{r_0} \fint_{B_{2r_0}} \abs{D^2u}+ \int_0^{r_0} \frac{\tilde \omega_{Df}(t)}{t}\,dt + \int_0^{r_0} \frac{\tilde \omega_{\rm lot}(t)}{t}\,dt\right\},
\]
where $C=C(d, \lambda, \Lambda, \omega_{\rm coef},\beta)$.

To analyze the modulus of continuity $\varrho_{\rm lot}$, recall the definition:
\[
\omega_{\rm lot}(r) :=\norm{Du}_{L^\infty(B_r)} \left\{\omega_{D \vec b}(r) +r^{1-d} \sup_{x\in \Omega} \int_{\Omega \cap B_r(x)} \abs{Dc} \right\}+ \norm{u}_{L^\infty(B_r)}\,\omega_{Dc}(r).
\]
We isolate a component of this expression by defining:
\[
\omega_{Dc,1}(r) := r^{1-d} \sup_{x\in \Omega} \int_{\Omega \cap B_r(x)} \abs{Dc}.
\]
It is clear that $\omega_{Dc,1}(r)$ satisfies the Dini condition.
Then, for $r \in (0,r_0/2)$, we obtain the following estimate for $\varrho_{\rm lot}(r)$:
\begin{equation}		\label{eq2119sun}
\varrho_{\rm lot}(r) \le C \norm{Du}_{L^\infty(B_{r_0})} \left\{
\varrho_{D\vec b}(r)+ \hat \varrho_{Dc}(r)\right\}
 +C \norm{u}_{L^\infty(B_{r_0})} \varrho_{Dc}(r),
\end{equation}
where the moduli of continuity are defined by
\begin{align*}
\varrho_{D\vec b}(r)&:=C \int_0^{2r/\kappa} \frac{\tilde \omega_{D\vec b}(t)}{t}\,dt,\\
\varrho_{Dc}(r)&:=C \int_0^{2r/\kappa} \frac{\tilde \omega_{Dc}(t)}{t}\,dt,\\
\hat \varrho_{Dc}(r)&:=C \int_0^{2r/\kappa} \frac{\tilde \omega_{Dc,1}(t)}{t}\,dt.
\end{align*}
Note that inequality \eqref{eq2119sun} implies that if $D\vec b$ and $Dc$ are $C^\alpha$ functions, then choosing $\beta \in (\alpha,1)$ yields $\varrho_{\rm lot}(t) \lesssim t^\alpha$.
The theorem is proved.
\qed


\end{document}